\documentclass[11pt,oneside]{article}
\usepackage{color}
\usepackage{psfrag}
\usepackage{epsfig}
\usepackage{graphicx}
\usepackage{amsmath}
\usepackage{amssymb}
\usepackage{amsthm}
\usepackage{float}
\usepackage{cite}
\usepackage[latin1]{inputenc}
\usepackage{cancel}

\newtheorem{theorem}{Theorem}[section]

\newtheorem{proposition}[theorem]{Proposition}
\newtheorem{lemma}[theorem]{Lemma}

\newtheorem{remark}[theorem]{Remark}
\newcommand{\F}{{\mathbb F}}

\newcommand{\fqs}{{\mathbb F_{q^2}}}
\newcommand{\PGU}{{\rm PGU}}
\newcommand{\PSU}{{\rm PSU}}
\newcommand{\PGL}{{\rm PGL}}
\newcommand{\PSL}{{\rm PSL}}
\newcommand{\SL}{{\rm SL}}
\newcommand{\PG}{{\rm PG}}
\newcommand{\cX}{{\mathcal X}}

\newcommand{\cH}{{\mathcal H}}

\newcommand{\aut}{{\rm Aut}}
\newcommand{\ord}{{\rm ord}}
\newcommand{\diag}{{\rm diag}}

\begin{document}

\title{On the classification problem for the genera of quotients of the Hermitian curve}
\date{}
\author{Francesca Dalla Volta, Maria Montanucci, Giovanni Zini
\thanks{
This research was partially supported by Ministry for Education, University and Research of Italy (MIUR) and by the Italian National Group for Algebraic and Geometric Structures and their Applications (GNSAGA - INdAM). 
The research of Francesca Dalla Volta was partially supported by the Project PRIN 2015 'Teoria dei gruppi e applicazioni' - Prot. N. 2015TW9LSR. 
URL: Francesca Dalla Volta (francesca.dallavolta@unimib.it), Maria Montanucci (maria.montanucci@unibas.it), Giovanni Zini (giovanni.zini@unimib.it)}
}

\maketitle

\begin{abstract}
In this paper we characterize the genera of those quotient curves $\cH_q/G$ of the $\mathbb{F}_{q^2}$-maximal Hermitian curve $\cH_q$ for which $G$ is contained in the maximal subgroup $\mathcal{M}_1$ of $\aut(\cH_q)$ fixing a self-polar triangle, or $q$ is even and $G$ is contained in the maximal subgroup $\mathcal{M}_2$ of $\aut(\cH_q)$ fixing a pole-polar pair $(P,\ell)$ with respect to the unitary polarity associated to $\cH_q$. In this way several new values for the genus of a maximal curve over a finite field are obtained. These results together with \cite{MZultimo} leave just two open cases to provide the complete list of genera of Galois subcovers of the Hermitian curve; namely, the open cases in \cite{BMXY} when $G$ fixes a point $P \in \cH_q(\mathbb{F}_{q^2})$ and $q$ is even, and the open cases in \cite{MZ} when $G\leq\mathcal{M}_2$ and $q$ is odd.
\end{abstract}



\section{Introduction}
Let $q$ be a power of a prime $p$, $\mathbb{F}_{q^2}$ be the finite field with $q^2$ elements, and $\cX$ be a projective, absolutely irreducible, non-singular algebraic curve defined over $\mathbb{F}_{q^2}$.
The curve $\cX$ is called $\mathbb{F}_{q^2}$-maximal if the number $|\cX(\mathbb{F}_{q^2})|$ of its $\mathbb{F}_{q^2}$-rational points attains the Hasse-Weil upper bound $q^2+1+2gq$, where $g$ denotes the genus of $\cX$. Maximal curves have been investigated for their applications in Coding Theory; see \cite{Sti,vdG2}. Surveys on maximal curves are found in \cite{FT,G,G2,GS,vdG} and \cite[Chapter 10]{HKT}.

A well-known and important example of an $\mathbb{F}_{q^2}$-maximal curve is the Hermitian curve $\cH_q$. It is defined as any $\mathbb{F}_{q^2}$-rational curve which is projectively equivalent to the plane curve $X^{q+1}+Y^{q+1}+Z^{q+1}=0$. For fixed $q$, the curve $\cH_q$ has the largest genus $g(\cH_q)=q(q-1)/2$ that an $\mathbb{F}_{q^2}$-maximal curve can have. The automorphism group $\aut(\cH_q)$ is defined over $\mathbb{F}_{q^2}$ and isomorphic to the projective general unitary group $\PGU(3,q)$, the group of projectivities of ${\rm PG}(2,q^2)$ commuting with the unitary polarity associated with $\cH_q$. The automorphism group $\aut(\cH_q)$ has size bigger than $16g(\cH_q)^4$, while any curve of genus $g\geq2$ not isomorphic to the Hermitian curve has less than $16g^4$.

By a result commonly referred to as the Kleiman-Serre covering result (see \cite{KS} and \cite[Proposition 6]{L}) a curve $\cX$ defined over $\mathbb{F}_{q^2}$ and $\mathbb{F}_{q^2}$-covered  by an $\mathbb{F}_{q^2}$-maximal curve is $\mathbb{F}_{q^2}$-maximal as well. In particular, $\mathbb{F}_{q^2}$-maximal curves can be obtained as Galois $\mathbb{F}_{q^2}$-subcovers of an $\mathbb{F}_{q^2}$-maximal curve $\cX$, that is, as quotient curves $\cX/G$ where $G$ is a finite $\mathbb{F}_{q^2}$-automorphism group of $\cX$. Most of the known maximal curves are Galois covered by the Hermitian curve; see for instance \cite{GSX,CKT2,GHKT2,MZ,MZRS} and the references therein.

A challenging open problem is the determination of the spectrum $\Gamma(q^2)$ of genera of $\mathbb{F}_{q^2}$-maximal curves, for given $q$; some values in $\Gamma(q^2)$ arise from $\mathbb{F}_{q^2}$-maximal curves which are not covered or Galois covered by the Hermitian curve.
The first example of a maximal curve not Galois covered by the Hermitian curve was discovered by Garcia and Stichtenoth \cite{GS3}. This curve is $\mathbb{F}_{3^6}$-maximal and it is not Galois covered by $\cH_{27}$. It is a special case of the $\mathbb{F}_{q^6}$-maximal GS curve, which was later shown not to be Galois covered by $\cH_{q^3}$ for any $q>3$, \cite{GMZ,Mak}. Giulietti and Korchm\'aros \cite{GK} provided an $\mathbb{F}_{q^6}$-maximal curve, nowadays referred to as the GK curve, which is not covered by the Hermitian curve $\cH_{q^3}$ for any $q>2$, like some of its Galois subcovers \cite{GQZ,TTT}. Two generalizations of the GK curve were introduced by Garcia, G\"uneri and Stichtenoth \cite{GGS} and by Beelen and Montanucci in \cite{BM}. Both these two generalizations are $\mathbb{F}_{q^{2n}}$-maximal curves, for any $q$ and odd $n \geq 3$. Also, they are not Galois covered by the Hermitian curve $\cH_{q^n}$ for $q>2$ and $n \geq 5$, see \cite{DM,BM}; the Garcia-G\"uneri-Stichtenoth's generalization is also not Galois covered by $\cH_{2^n}$ for $q=2$, see \cite{GMZ}.

Apart from the examples listed above, most of the known values in $\Gamma(q^2)$ have been obtained from quotients curves $\cH_q/G$ of the Hermitian curve, which have beed investigated in many papers.
Therefore, towards the determination of $\Gamma(q^2)$ it is important to study the genera $g(\cH_q/G)$ for all $G\leq\aut(\cH_q)$.
We list some known partial results:
\begin{itemize}
\item $g(\cH_q/G)$ for some $G$ fixing a point of $\cH_q(\mathbb{F}_{q^2})$; see \cite{BMXY,GSX,AQ}.
\item $g(\cH_q/G)$ when $G$ normalizes a Singer subgroup of $\aut(\cH_q)$ fixing three non-collinear points of $\cH_q(\mathbb{F}_{q^6})\setminus\cH_q(\mathbb{F}_{q^2})$; see \cite{GSX, CKT1,CKT2}.
\item $g(\cH_q/G)$ when $G$ has prime order; see \cite{CKT2}.
\item $g(\cH_q/G)$ when $G$ fixes neither points nor triangles in ${\rm PG}(2,\bar{\mathbb{F}}_{q^2})$; see \cite{MZultimo}. 
\end{itemize}
Only the following cases for $G\leq\aut(\cH_q)$ still have to be analyzed:
\begin{enumerate}
\item $G$ fixes a self-polar triangle in ${\rm PG}(2,q^2)$,
\item $G$ fixes an $\mathbb{F}_{q^2}$-rational point $P\notin\cH_q$.
\item $G$ fixes a point $P \in \cH_q(\mathbb{F}_{q^2})$, $p=2$ and $|G|=p^\ell d$ where $p^\ell \leq q$ and $d \mid (q-1)$.
\end{enumerate}
We observe that some partial results are known in the literature for the cases 1 and 2.
In this paper a complete analysis of Case 1 is given, as well as a complete analysis of Case 2 when $p=2$.
More precisely, this paper is organized as follows. Section \ref{sec:preliminari} recalls necessary preliminary results on the Hermitian curve and its automorphism group. 
In Section \ref{sec:triangolo} a complete analysis of Case 1 is given. 
Section \ref{sec:polopolare} contains a complete analysis of Case 2 under the assumption $p=2$.
Finally, Section \ref{sec:nuovigeneri} collects some new genera of maximal curves obtained for fixed values of $q$.

\section{Preliminary results}\label{sec:preliminari}

Throughout the paper $p$ is a prime number, $n$ is a positive integer, $q=p^n$, and $\mathbb{F}_q$ is the finite field with $q$ elements. The Deligne-Lusztig curves defined over $\mathbb{F}_q$ were originally introduced in \cite{DL}. Other than the projective line, there are three families of Deligne-Lusztig curves, named Hermitian curves, Suzuki curves and Ree curves. The Hermitian curve $\mathcal H_q$ arises from the algebraic group $^2A_2(q)={\rm PGU}(3,q)$ of order $(q^3+1)q^3(q^2-1)$, that is, $\aut(\cH_q)=\PGU(3,q)$. It has genus $q(q-1)/2$ and is $\mathbb F_{q^2}$-maximal. Three $\mathbb{F}_{q^2}$-isomorphic models of $\cH_q$ are given by the following equations:
\begin{equation} \label{M1}
X^{q+1}+Y^{q+1}+Z^{q+1}=0;
\end{equation}
\begin{equation} \label{M2}
Y^{q}Z+YZ^{q}- X^{q+1}=0;
\end{equation}
\begin{equation} \label{M3}
XY^{q}-X^{q}Y+\omega Z^{q+1}=0,
\end{equation}
where $\omega\in\mathbb{F}_{q^2}$ satisfies $\omega^{q+1}=-1$.
The curves \eqref{M1} and \eqref{M2} are known as the Fermat and the Norm-Trace model of the Hermitian curve, respectively.
The action of $\aut(\cH_q)$ on the set $\cH_q(\mathbb{F}_{q^2})$ of all $\mathbb{F}_{q^2}$-rational points of $\cH_q$ is equivalent to the $2$-transitive permutation representation of $\PGU(3,q)$ on the isotropic points with respect to the unitary form.
The combinatorial properties of $\cH_q(\mathbb{F}_{q^2})$ can be found in \cite{HP}. The size of $\cH_q(\mathbb{F}_{q^2})$ is equal to $q^3+1$, and a line of $PG(2,q^2)$ has either $1$ or $q+1$ common points with $\cH_q(\mathbb{F}_{q^2})$, that is, it is either a tangent line or a chord of $\cH_q(\mathbb{F}_{q^2})$.


The following classification of maximal subgroups of the projective special subgroup $\PSU(3,q)$ of $\PGU(3,q)$ goes back to Mitchell and Hartley; see \cite{M}, \cite{H}, \cite{HO}.
Actually, in \cite{KL} it is possible to find a classification of maximal subgroups of $\PSU(n,q)$, $n\geq 3$, in the context of Aschbacher's classes of subgroups, after the theorem of classification for finite simple groups.
Here we prefer to go back to the results of Mitchell and Hartley where it is more readable the relation between subgroups and action on $\PG(2,\bar{\mathbb{F}}_{q^2})$ and $\cH_q$.

\begin{theorem} \label{Mit} Let $d={\rm gcd}(3,q+1)$. Up to conjugacy, the subgroups below give a complete list of maximal subgroups of $\PSU(3,q)$.
\begin{itemize}
\item[(i)] the stabilizer of an $\F_{q^2}$-rational point of $\cH_q$. It has order $q^3(q^2-1)/d$;
\item[(ii)] the stabilizer of an $\F_{q^2}$-rational point off $\cH_q$ $($equivalently the stabilizer of a chord of $\cH_q(\mathbb{F}_{q^2}))$. It has order $q(q-1)(q+1)^2/d$;
\item[(iii)] the stabilizer of a self-polar triangle with respect to the unitary polarity associated to $\cH_q(\mathbb{F}_{q^2})$. It has order $6(q+1)^2/d$;
\item[(iv)] the normalizer of a (cyclic) Singer subgroup. It has order $3(q^2-q+1)/d$ and preserves a triangle
in $\PG(2,q^6)\setminus\PG(2,q^2)$ left invariant by the Frobenius collineation $\Phi_{q^2}:(X,Y,T)\mapsto (X^{q^2},Y^{q^2},T^{q^2})$ of $\PG(2,\bar{\mathbb{F}}_{q})$;

{\rm for $p>2$:}
\item[(v)] ${\rm PGL}(2,q)$ preserving a conic;
\item[(vi)] $\PSU(3,p^m)$ with $m\mid n$ and $n/m$ odd;
\item[(vii)] subgroups containing $\PSU(3,p^m)$ as a normal subgroup of index $3$, when $m\mid n$, $n/m$ is odd, and $3$ divides both $n/m$ and $q+1$;
\item[(viii)] the Hessian groups of order $216$ when $9\mid(q+1)$, and of order $72$ and $36$ when $3\mid(q+1)$;
\item[(ix)] ${\rm PSL(2,7)}$ when $p=7$ or $-7$ is not a square in $\mathbb{F}_q$;
\item[(x)] the alternating group $\mathbf{A}_6$ when either $p=3$ and $n$ is even, or $5$ is a square in $\mathbb{F}_q$ but $\mathbb{F}_q$ contains no cube root of unity;
\item[(xi)] an extension of order $720$ of $\mathbf{A}_6$, when $p=5$ and $n$ is odd;
\item[(xii)] the alternating group $\mathbf{A}_7$ when $p=5$ and $n$ is odd;

{\rm for $p=2$:}
\item[(xiii)] $\PSU(3,2^m)$ with $m\mid n$ and $n/m$ an odd prime;
\item[(xiv)] subgroups containing $\PSU(3,2^m)$ as a normal subgroup of index $3$, when $n=3m$ with $m$ odd;
\item[(xv)] a group of order $36$ when $n=1$.
\end{itemize}
\end{theorem}

We also need the classification of subgroups of the projective general linear group $\PGL(2,q)$:

\begin{theorem}\label{Di}{\rm (see \cite[Chapt. XII, Par. 260]{D}, \cite[Kap. II, Hauptsatz 8.27]{Hup}, \cite[Thm. A.8]{HKT})}
The following is the complete list of subgroups of $\PGL(2,q)$ up to conjugacy:
\begin{itemize}
\item[(i)] the cyclic group of order $h$ with $h\mid(q\pm1)$;
\item[(ii)] the elementary abelian $p$-group of order $p^f$ with $f\leq n$;
\item[(iii)] the dihedral group of order $2h$ with $h\mid(q\pm1)$;
\item[(iv)] the alternating group $\mathbf{A}_4$ for $p>2$, or $p=2$ and $n$ even;
\item[(v)] the symmetric group $\mathbf{S}_4$ for $16\mid(q^2-1)$;
\item[(vi)] the alternating group $\mathbf{A}_5$ for $p=5$ or $5\mid(q^2-1)$;
\item[(vii)] the semidirect product of an elementary abelian $p$-group of order $p^f$ by a cyclic group of order $h$, with $f\leq n$ and $h\mid\gcd(p^f-1,q-1)$;
\item[(viii)] $\PSL(2,p^f)$ for $f\mid n$;
\item[(ix)] $\PGL(2,p^f)$ for $f\mid n$.
\end{itemize}
\end{theorem}

In our investigation it is useful to know how an element of $\PGU(3,q)$ of a given order acts on $\PG(2,\bar{\mathbb{F}}_{q^2})$, and in particular on $\cH_q$. This can be obtained as a corollary of Theorem \ref{Mit}, and is stated in Lemma $2.2$ with the usual terminology of collineations of projective planes; see \cite{HP}. In particular, a linear collineation $\sigma$ of $\PG(2,\bar{\mathbb{F}}_{q^2})$ is a $(P,\ell)$-\emph{perspectivity} if $\sigma$ preserves  each line through the point $P$ (the \emph{center} of $\sigma$), and fixes each point on the line $\ell$ (the \emph{axis} of $\sigma$). A $(P,\ell)$-perspectivity is either an \emph{elation} or a \emph{homology} according to $P\in \ell$ or $P\notin\ell$, respectively. A $(P,\ell)$-perspectivity is in  $\PGL(3,q^2)$ if and only if its center and its axis are in $\PG(2,\mathbb{F}_{q^2})$.

Lemma \ref{classificazione} gives a classification of nontrivial elements in $\PGU(3,q)$, we will refer to throughout the paper.
\begin{lemma}\label{classificazione}{\rm (\!\!\cite[Lemma 2.2]{MZRS})}
For a nontrivial element $\sigma\in\PGU(3,q)$, one of the following cases holds.
\begin{itemize}
\item[(A)] ${\rm ord}(\sigma)\mid(q+1)$ and $\sigma$ is a homology, whose center $P$ is a point of $\PG(2,q^2)\setminus\cH_q$ and whose axis $\ell$ is a chord of $\cH_q(\mathbb{F}_{q^2})$ such that $(P,\ell)$ is a pole-polar pair with respect to the unitary polarity associated to $\cH_q(\mathbb{F}_{q^2})$.
\item[(B)] $p\nmid{\rm ord}(\sigma)$ and $\sigma$ fixes the vertices $P_1$, $P_2$, $P_3$ of a non-degenerate triangle $T\subset\PG(2,q^6)$.
\begin{itemize}
\item[(B1)] ${\rm ord}(\sigma)\mid(q+1)$, $P_1,P_2,P_3\in\PG(2,q^2)\setminus\cH_q$, and $T$ is self-polar with respect to the unitary polarity associated to $\cH_q(\mathbb{F}_{q^2})$.
\item[(B2)] ${\rm ord}(\sigma)\mid(q^2-1)$, ${\rm ord}\nmid(q+1)$, $P_1\in\PG(2,q^2)\setminus\cH_q$, and $P_2,P_3\in\cH_q(\fqs)$. 
\item[(B3)] ${\rm ord}(\sigma)\mid(q^2-q+1)$, and $P_1,P_2,P_3\in\cH_q(\mathbb{F}_{q^6})\setminus\cH_q(\mathbb{F}_{q^2})$.
\end{itemize}
\item[(C)] ${\rm ord}(\sigma)=p$ and $\sigma$ is an elation, whose center $P$ is a point of $\cH_q(\fqs)$ and whose axis $\ell$ is tangent to $\cH_q$ at $P$ such that $(P,\ell)$ is a pole-polar pair with respect to the unitary polarity associated to $\cH_q(\mathbb{F}_{q^2})$.
\item[(D)] wither ${\rm ord}(\sigma)=p$ with $p\ne2$, or ${\rm ord}(\sigma)=4$ with $p=2$; $\sigma$ fixes a point $P\in\cH_q(\fqs)$ and a line $\ell$ which is tangent to $\cH_q$ at $P$, such that $(P,\ell)$ is a pole-polar pair with respect to the unitary polarity associated to $\cH_q(\mathbb{F}_{q^2})$.
\item[(E)] ${\rm ord}(\sigma)=p\cdot d$, where $1\ne d\mid(q+1)$; $\sigma$ fixes two points $P\in\cH_q(\fqs)$ and $Q\in\PG(2,q^2)\setminus\cH_q$; $\sigma$ fixes the line $PQ$ which is the tangent to $\cH_q$ at $P$, and another line through $P$ which is the polar of $Q$.
\end{itemize}
\end{lemma}

From Function Field Theory we need the Riemann-Hurwitz formula; see \cite[Theorem 3.4.13]{Sti}. Every subgroup $G$ of $\PGU(3,q)$ produces a quotient curve $\cH_q/G$, and the cover $\cH_q\rightarrow\cH_q/G$ is a Galois cover defined over $\mathbb{F}_{q^2}$. Also, the degree $\Delta$ of the different divisor is given by the Riemann-Hurwitz formula, namely
$$\Delta=(2g(\cH_q)-2)-|G|(2g(\cH_q/G)-2).$$
On the other hand, $\Delta=\sum_{\sigma\in G\setminus\{id\}}i(\sigma)$ where $i(\sigma)\geq0$ is given by the Hilbert's different formula \cite[Thm. 3.8.7]{Sti}, namely
\begin{equation}\label{contributo}
\textstyle{i(\sigma)=\sum_{P\in\cH_q(\bar\F_{q^2})}v_P(\sigma(t)-t),}
\end{equation}
where $v_P(\cdot)$ denotes the valuation function at $P$ and $t$ is a local parameter at $P$, i.e. $v_P(t)=1$.

By analyzing the geometric properties of the elements $\sigma \in \PGU(3,q)$, it turns out that there are only a few possibilities for $i(\sigma)$.
This is obtained as a corollary of Lemma \ref{classificazione}.

\begin{theorem}\label{caratteri}{\rm (\!\!\cite[Theorem 2.7]{MZRS})}
For a nontrivial element $\sigma\in \PGU(3,q)$, one of the following cases occurs.
\begin{enumerate}
\item If $\ord(\sigma)=2$ and $2\mid(q+1)$, then $\sigma$ is of type {\rm(A)} and $i(\sigma)=q+1$.
\item If $\ord(\sigma)=3$, $3 \mid(q+1)$, and $\sigma$ is of type {\rm(B3)}, then $i(\sigma)=3$.
\item If $\ord(\sigma)\ne 2$, $\ord(\sigma)\mid(q+1)$, and $\sigma$ is of type {\rm(A)}, then $i(\sigma)=q+1$.
\item If $\ord(\sigma)\ne 2$, $\ord(\sigma)\mid(q+1)$, and $\sigma$ is of type {\rm(B1)}, then $i(\sigma)=0$.
\item If $\ord(\sigma)\mid(q^2-1)$ and $\ord(\sigma)\nmid(q+1)$, then $\sigma$ is of type {\rm(B2)} and $i(\sigma)=2$.
\item If $\ord(\sigma)\ne3$ and $\ord(\sigma)\mid(q^2-q+1)$, then $\sigma$ is of type {\rm(B3)} and $i(\sigma)=3$.
\item If $p=2$ and $\ord(\sigma)=4$, then $\sigma$ is of type {\rm(D)} and $i(\sigma)=2$.
\item If $p\ne2$, $\ord(\sigma)=p$, and $\sigma$ is of type {\rm(D)}, then $i(\sigma)=2$.
\item If $\ord(\sigma)=p$ and $\sigma$ is of type {\rm(C)}, then $i(\sigma)=q+2$.
\item If $\ord(\sigma)=p\cdot d$ where $d$ is a nontrivial divisor of $q+1$, then $\sigma$ is of type {\rm(E)} and $i(\sigma)=1$.
\end{enumerate}
\end{theorem}





\begin{remark}\label{abuso}
With a small abuse of notation, in the paper we denote by $\sigma$ the element $\bar{\sigma}=\sigma Z({\rm GU}(3,q))\in{\rm GU}(3,q)/Z({\rm GU}(3,q))=\PGU(3,q)$.
\end{remark}

\section{$G$ stabilizes a self-polar triangle $T\subset {PG}(2,q^2)\setminus\cH_q$}\label{sec:triangolo}

In this section we take $G$ to be contained in the maximal subgroup $M$ of $\PGU(3,q)$ stabilizing a self-polar triangle $T=\{P_1,P_2,P_3\}$ with respect to the unitary polarity associated to $\cH_q(\fqs)$.
Recall that $M\cong (C_{q+1}\times C_{q+1})\rtimes \mathbf{S}_3$, where $C_{q+1}\times C_{q+1}$ is the pointwise stabilizer of $T$ and $\mathbf{S}_3$ acts faithfully on $T$. We denote by $G_T=G\cap(C_{q+1}\times C_{q+1})$ the pointwise stabilizer of $T$ in $G$.

We use the plane model \eqref{M1} of $\cH_q$.
Up to conjugation in ${\rm PGU}(3,q)$, we can assume that $T$ is the fundamental triangle with $P_1=(1,0,0)$, $P_2=(0,1,0)$, $P_3=(0,0,1)$.
For each $\sigma\in C_{q+1}\times C_{q+1}$ we will consider as a representative of $\sigma$ the diagonal matrix having $(q+1)$-th roots of unity in the diagonal entries and $1$ in the third row and column.
We can choose as complement $\mathbf{S}_3$ the monomial group made by matrices with exactly one nonzero entry equal to $1$ in each row and column; any other complement $\mathbf{S}_3$ has $(q+1)$-th roots of unity as nonzero entries.

Partial results when the order of $G$ is coprime to $p$ were obtained in \cite{RobinHood} for the Fermat curve $\mathcal{F}_m : X^{m}+Y^{m}+Z^{m}=0$ with $p \nmid m$. 

Each result in Sections \ref{sec:punt} to \ref{sec:indice6} is stated as follows: firstly, sufficient numerical conditions for the existence of groups $G$ with a certain action on $T$ are given, and the genus of $\cH_q/G$ is provided; secondly, it is shown that groups $G$ with that given action on $T$ satisfy those numerical conditions.
In our investigation, the numerical conditions arise as necessary by the analysis of a putative group $G$ with a given action on $T$, and in a second moment they also appear as sufficient; neverthless, we choose to start the exposition of each result with the numerical conditions, showing firstly that they are sufficient and then that they are necessary, in order to make the proofs shorter and to highlight the purely numerical characterization that we obtained.

\subsection{$G$ stabilizes $T$ pointwise}\label{sec:punt}

In this section $G=G_T$, that is, $G\leq C_{q+1}\times C_{q+1}$.

\begin{theorem}\label{fissatorepuntuale}
Let $q+1=\prod_{i=1}^{\ell}p_i^{r_i}$ be the prime factorization of $q+1$.
\begin{itemize}
\item[(i)] For any divisors $a=\prod_{i=1}^{\ell}p_i^{s_i}$ and $b=\prod_{i=1}^{\ell}p_i^{t_i}$ of $q+1$ $($$0\leq s_i,t_i\leq r_i$$)$, let $c=\prod_{i=1}^{\ell}p_i^{u_i}$ be such that, for all $i=1,\ldots,\ell$, we have $u_i=\min\{s_i,t_i\}$ if $s_i\ne t_i$, and $s_i\leq u_i\leq r_i$ if $s_i=t_i$.
Define $d=a+b+c-3$.
Let $e=\frac{abc}{\gcd(a,b)}\cdot\prod_{i=1}^{\ell}p_i^{v_i}$, where for all $i$'s $v_i$ satisfies $0\leq v_i\leq r_i-\max\{s_i,t_i,u_i\}$.
We also require that, if $p_i=2$ and either $2\nmid abc$ or $2\mid\gcd(a,b,c)$, then $v_i=0$.
Then there exists a subgroup $G$ of $C_{q+1}\times C_{q+1}$ such that
\begin{equation}\label{generefissatore}
 g(\cH_q/G)=\frac{(q+1)(q-2-d)+2e}{2e}.
\end{equation}
\item[(ii)] Conversely, if $G\leq C_{q+1}\times C_{q+1}$, then the genus of $\cH_q/G$ is given by Equation \eqref{generefissatore}, where $e=|G|$ and $d$ is the number of elements of type {\rm (A)} in $G$; $d,e$ satisfy the numerical assumptions in point {\it (i)}, for some $a,b,c$.
\end{itemize}
\end{theorem}

\begin{proof}
If $G\leq C_{q+1}\times C_{q+1}$ has order $e$ and contains exactly $d$ homologies, then any other $\sigma\in G$ satisfies $i(\sigma)=0$ by Theorem \ref{caratteri}; hence, Equation \eqref{generefissatore} follows from the Riemann-Hurwitz formula.
Thus, the claim is proved once we characterize the possible values of $d$ and $e$.
To this aim, we first study the subgroup $K$ of $G$ which is generated by the homologies of $G$; then we characterize the order of $G/K$ such that $G$ contains no homologies out of $K$.

{\it (i)}: let the parameters $a,b,c,e$ be as in point {\it (i)}.
Define $A=\{\diag(\lambda,1,1)\mid\lambda^{a}=1\}$, $B=\{\diag(1,\lambda,1)\mid\lambda^{b}=1\}$, and $C=\{\diag(\lambda,\lambda,1)\mid\lambda^{c}=1\}$.
For any $i=1,\ldots,\ell$, let $m_i=\max\{s_i,t_i,u_i\}$, and $\lambda_i,\mu_i\in\mathbb{F}_{q^2}$ have order $o(\lambda_i)=p_i^{v_i}$ and $o(\mu_i)=p_i^{v_i+m_i}$. Define $H_i\leq C_{q+1}\times C_{q+1}$ by
$$ H_i=\begin{cases} \langle\diag(\mu_i,\lambda_i,1)\rangle & \textrm{if }\; m_i=s_i>0; \\
\langle\diag(\lambda_i,\mu_i,1)\rangle & \textrm{if }\; m_i\ne s_i,\, m_i=t_i>0 ;\\
\langle\diag(\lambda_i\mu_i,\lambda_i^{-1}\mu_i,1)\rangle & \textrm{otherwise.}   \end{cases} $$
Now choose $G=ABC\cdot\prod_{i=1}^{\ell}H_i$.
It is not difficult to check that the homologies in $G$ are exactly the nontrivial elements of $(A\cup B\cup C)\setminus\{id\}$ and their number is $d=a+b+c-3$. Also, the order of $K=ABC$ is $\frac{abc}{\gcd(a,b)}$. The order of $\frac{G}{K}$ is $\prod_{i=1}^{\ell}p_i^{v_i}$, since $|H_i|=p_i^{v_i+m_i}$ and $|H_i\cap K|=p_i^{m_i}$.
Hence, $G$ has order $e$ and contains exactly $d$ homologies; the first part of the claim is then proved.

{\it (ii)}: let $G\leq C_{q+1}\times C_{q+1}$ with $|G|=e$.
Define the subgroups $A=\{\sigma\in G\mid\sigma=\diag(\lambda,1,1),\lambda^{q+1}=1\}$, $B=\{\sigma\in G\mid\sigma=\diag(1,\mu,1),\mu^{q+1}=1\}$, and $C=\{\sigma\in G\mid\sigma=\diag(\nu,\nu,1),\nu^{q+1}=1\}$ of $G$, of order $a=\prod_{i=1}^{\ell}p_i^{s_i}$, $b=\prod_{i=1}^{\ell}p_i^{t_i}$, and $c=\prod_{i=1}^{\ell}p_i^{u_i}$, respectively. Since $|A|$, $|B|$, and $|C|$ divide $q+1$, we have $s_i,t_i,u_i\leq r_i$.
The homologies of $G$ are exactly the nontrivial elements of $A\cup B\cup C$, and their number is $d=a+b+c-3$.
We have $\gcd(a,b)\mid c$, so that $\min\{s_i,t_i\}\mid u_i$. If $s_i\ne t_i$, say for instance $s_i<t_i$, then $u_i=s_i$; otherwise, $BC$ would contain a subgroup of $A$ of order $p_i^{\min\{t_i,u_i\}}>p_i^{s_i}$, a contradiction. Thus, $a,b,c$ satisfy the restrictions of point {\it (i)}.
Consider the group $G/K$, where $K=ABC$ has order $\frac{abc}{\gcd(a,b)}$.
Let $i\in\{1,\ldots,\ell\}$ be such that $v_i>0$, and $P_i$ be a Sylow $p_i$-subgroup of $G/K$, of order $|P_i|=p_i^{v_i}$.
\begin{itemize}
\item Suppose $p_i\nmid abc$ and $p_i\ne2$. Then $P_i$ is also a Sylow $p_i$-subgroup of $G$. If $v_i>r_i$, then $P_i$ is not cyclic and contains a subgroup $C_{p_i}\times C_{p_i}$; hence $|P_i\cap K|\geq p_i^2$, a contradiction. Hence, $v_i\leq r_i=r_i-\max\{s_i,t_i,u_i\}$.
\item Suppose $p_i=2\nmid abc$. Then $P_i$ is a Sylow $2$-subgroup of $G$ having trivial intersection with $K$ and containing an involution $\sigma$. By Theorem \ref{caratteri}, $\sigma$ is a homology, a contradiction.
\item Suppose $p_i\mid abc$ and $p_i\ne2$. Let $\alpha K\in G/K$ be a $p_i$-element, say $o(\alpha K)=p_i^{f}$. We can assume that $\alpha$ is a $p_i$-element. In fact, if $\alpha^{p_i^{f}}=\beta \gamma$ with $\beta$ a $p_i$-element and $p_i\nmid o(\gamma)$, then choose $\bar\gamma\in K$ with ${\bar\gamma}^{p_i^{f}}=\gamma^{-1}$ and replace $\alpha$ with $\alpha\bar\gamma$; now, $\alpha^{p_i^f}=\beta$ where $\beta\in K$ is a $p_i$-element.

We show that $o(\beta)=p_i^{m_i}$, where $m_i=\max\{s_i,t_i,u_i\}$.
Let $o(\beta)=p_i^{n_i}$, and suppose by contradiction that $n_i<m_i$.
If $m_i=\min\{s_i,t_i,u_i\}$, that is $s_i=t_i=u_i$, then $K$ contains all $p_i$-elements of $C_{q+1}\times C_{q+1}$ of order smaller than or equal to $p_i^{m_i}$, a contradiction to $\alpha^{p_i^{f-1}}\notin K$.
If $m_i>\min\{s_i,t_i,u_i\}$, then $K$ contains all $p_i$-elements of $C_{q+1}\times C_{q+1}$ of order smaller than or equal to $p_i^{\min\{s_i,t_i,u_i\}}$, and hence $m_i>n_i>\min\{s_i,t_i,u_i\}$.
Let $\delta\in K$ have order $o(\delta)=p_i^{m_i}$. Then the elements $\delta^{p_i^{m_i-n_i-1}}\in K$ and $\alpha^{p_i^{f-1}}\notin K$ satisfy $o(\delta^{p_i^{m_i-n_i-1}})=o(\alpha^{p_i^{f-1}})=p_i^{n_i+1}$, $\delta^{p_i^{m_i-n_i-1}}\notin\langle\alpha^{p_i^{f-1}}\rangle$, and $\alpha^{p_i^{f-1}}\notin\langle\delta^{p_i^{m_i-n_i-1}}\rangle$. This implies that $\langle\delta^{p_i^{m_i-n_i-1}},\alpha^{p_i^{f-1}}\rangle\cong C_{p_i^{n_i+1}}\times C_{p_i^{n_i+1}}$ by direct counting on the number of elements of order $p_i^{n_i+1}$. Thus, $G$ contains all $p_i$-elements of order at most $p_i^{n_i+1}$, a contradiction to $n_i>\min\{s_i,t_i,u_i\}$. Therefore, $o(\beta)=p_i^{m_i}$.

We show that $P_i$ is cyclic. Suppose by contradiction that this is not the case, so that $P_i$ has a subgroup $\alpha_1 K\times\alpha_2 K\cong C_{p_i}\times C_{p_i}$. As above, we can assume that $o(\alpha_1)=o(\alpha_2)=p_i^{m_i+1}$. Also, $\alpha_1\notin\langle\alpha_2\rangle$ and $\alpha_2\notin\langle\alpha_1\rangle$. Hence, $G$ has a subgroup $\langle\alpha_1,\alpha_2\rangle\cong C_{p_i^{m_i+1}}\times C_{p_i^{m_i+1}}$, a contradiction to $m_i+1>\min\{s_i,t_i,u_i\}$. Therefore, $P_i$ is cyclic.
Let $\alpha$ be a generator of $P_i$. As shown above, $\alpha^{p_i^{v_i}}$ has order $p_i^{m_i}$. Since the Sylow $p_i$-subgroup of $C_{q+1}\times C_{q+1}$ has exponent $p_i^{r_i}$, this proves that $0\leq v_i\leq r_i-m_i$.
\item Suppose $p_i=2\mid abc$ and $2\nmid\gcd(a,b,c)$. Then $2$ divides exactly one between $a$, $b$, and $c$. Arguing as in the previous point, one can show that the Sylow $2$-subgroup of $G$ is cyclic of order $2^{v_i+\max\{s_i,t_i,u_i\}}$ with $0\leq v_i\leq r_i-\max\{s_i,t_i,u_i\}$.
\item Suppose $p_i=2\mid abc$ and $2\mid\gcd(a,b,c)$.
Arguing as above, one can show that the Sylow $2$-subgroup $P_i$ of $G/K$ is cyclic and generated by $\alpha K$ where $\alpha\in G$ is a $2$-element; if $\alpha K$ has order $2^f$, then $\alpha^{2^{f}}$ has maximal order in the Sylow $2$-subgroup of $K$.
Let $\gamma=\alpha^{2^{f-1}}=\diag(\lambda,\mu,1)$.

Suppose $s_i=t_i=u_i$. If $o(\mu)\leq 2^{s_i}$, then $\sigma=\diag(1,\mu^{-1},1)\in K$ and $\gamma\sigma$ is a homology of order $2^{s_i+1}$, a contradiction; analogously if $o(\lambda)\leq 2^{s_i}$. Then $o(\lambda)=o(\mu)=2^{s_i+1}$, so that $\mu=\lambda^i$ with $i$ odd. Hence, $o(\lambda^{1-i})\leq 2^{s_i}$ and $\sigma=\diag(1,\lambda^{1-i},1)\in K$; thus, $\gamma\sigma$ is a homology of order $2^{s_i+1}$, a contradiction.

Suppose that $s_i,t_i,u_i$ are not equal, so that two of them are equal and smaller than the third one, say $s_i=t_i<u_i$. As shown above, this implies that $\gamma^2$ is a homology and $\gamma^2\in C$; otherwise, new homologies arise.
Then $\lambda^2=\mu^2$ and hence $\mu=-\lambda$, as $\gamma$ is of type (B1). Since $\sigma=\diag(1,-1,1)\in K$, $\gamma\sigma$ is a homology of order $2^{u_i+1}$ in $G$, a contradiction.
The argument is analogous if $\max\{s_i,t_i,u_i\}=s_i$ or $\max\{s_i,t_i,u_i\}=t_i$.
\end{itemize}
\end{proof}


\begin{remark}
From the proof of Theorem {\rm \ref{fissatorepuntuale}} it follows that the group-theoretic structure of $G$ is uniquely determined by the parameters $e=|G|$, $a=|\{\diag(\lambda,1,1)\in G\;\textrm{for some}\;\lambda\}|$, $b=|\{\diag(1,\lambda,1)\in G\;\textrm{for some}\;\lambda\}|$, and $c=|\{\diag(\lambda,\lambda,1)\in G\;\textrm{for some}\;\lambda\}|$.

In fact, write $G$ as a direct product $H_1\times\cdots\times H_{\ell}$ of its Sylow subgroups.
If $p_i$ divides at most one between $a$, $b$, and $c$, then $H_i$ is cyclic, of order $p_i^{v_i+\max\{s_i,t_i,u_i\}}$.
Otherwise, $p_i$ divides $a$, $b$, and $c$; in this case, $H_i$ is the direct product $C_1\times C_2$ of two cyclic groups of order $p_i^{v_i+\max\{s_i,t_i,u_i\}}$ and $p_i^{\min\{s_i,t_i,u_i\}}$.
\end{remark}

\subsection{The pointwise stabilizer of $T$ has index $2$ in $G$}\label{sec:indice2}

In this section $[G:G_T]=2$.
We consider separately the cases $q$ even and $q$ odd.

\begin{proposition}\label{indice2pari}
Let $q$ be even.
\begin{itemize}
\item[(i)] Let $a$, $c$, and $e$ be positive integers satisfying $e\mid(q+1)^2$, $c\mid(q+1)$, $a\mid c$, $ac\mid e$, $\frac{e}{a}\mid(q+1)$, and $\gcd\left(\frac{e}{ac},\frac{c}{a}\right)=1$.
Then there exists a subgroup $G\leq (C_{q+1}\times C_{q+1})\rtimes \mathbf{S}_3$ of order $2e$ such that $|G_T|=e$ and 
\begin{equation}\label{genereindice2}
g(\cH_q/G)=\frac{(q+1)\left(q-2a-c-\frac{e}{c}+1\right)+3e}{4e}.
\end{equation}
\item[(ii)] Conversely, if $G\leq(C_{q+1}\times C_{q+1})\rtimes \mathbf{S}_3$ and $G_T$ has index $2$ in $G$, then the genus of $\cH_q/G$ is given by Equation \eqref{genereindice2}, where: $e=|G|/2$; without loss of generality, $a-1$ is the number of homologies in $G$ with center $P_1$, which is equal to the number of homologies in $G$ with center $P_2$, and $c-1$ is the number of homologies in $G$ with center $P_3$; $a,c,e$ satisfy the numerical assumptions in point {\it (i)}.
\end{itemize}
\end{proposition}

\begin{proof}
{\it (i)}: let $a,c,e$ be positive integers satisfying the assumptions in point {\it (i)}.
Choose $\lambda,\mu,\rho\in\mathbb{F}_{q^2}$ such that $o(\lambda)=a$, $o(\mu)=c$, and $o(\rho)=e/a$.
Define $A=\langle\diag(\lambda,1,1)\rangle$, $C=\langle\diag(\mu,\mu,1)\rangle$, $D=\langle\diag(\rho,\rho^{-1},1)\rangle$, $S=ACD$, and
$$  \beta=\begin{pmatrix} 0 & \gamma & 0 \\ \gamma^{-1} & 0 & 0 \\ 0 & 0 & 1 \end{pmatrix}, $$
for some $\gamma\in\mathbb{F}_{q^2}$ with $\gamma^{q+1}=1$. Let $G=\langle S,\beta\rangle$. Note that $S\cong A\times CD$, $G\leq{\rm PGU}(3,q)$, $G$ stabilizes $T$, and $G_T=S$.
By direct checking, the assumptions imply that $\beta$ normalizes $S$, so that $G\cong S\rtimes \langle\beta\rangle$. The order of $G$ is $2e$.
We show that $g(\cH_q/G)$ is given by Equation \eqref{genereindice2}.

The homologies of $S$ are contained in $K=A\times C$ and their number is $(a-1)+(a-1)+(c-1)$ (note that $S$ has the role of $G$ in Theorem \ref{fissatorepuntuale}, where $b=a$). By Theorem \ref{caratteri}, we have $2a+c-3$ elements $\sigma\in S$ such that $i(\sigma)=q+1$; for any other $\sigma\in S$, $i(\sigma)=0$.
Let $\sigma\in S$, $\sigma=\diag(\delta,\epsilon,1)$.
If $\epsilon=\delta^{-1}$, then $o(\sigma\beta)=2$; by Theorem \ref{caratteri}, $i(\sigma\beta)=q+2$.
If $\epsilon\ne\delta^{-1}$, then $o(\sigma\beta)=2d$ where $d$ is a nontrivial divisor of $q+1$; by Theorem \ref{caratteri}, $i(\sigma\beta)=1$.

The number of elements $\sigma=\diag(\lambda\mu,\mu,1)\in K$ such that $\mu=(\lambda\mu)^{-1}$ is equal to $a$; it follows that the number of elements $\sigma=\diag(\delta,\epsilon,1)\in S$ such that $\epsilon=\delta^{-1}$ is equal to $a\cdot [S:K]=\frac{e}{c}$, because any element of $S/K$ is of type $\sigma K$ with $\sigma\in D$.
Thus, by the Riemann-Hurwitz formula, $q^2-q-2=2e(2g(\cH_q/G)-2)+\Delta$, where
\begin{equation}\label{boh}
\Delta= (2a+c-3)(q+1) + \frac{e}{c}(q+2) + \left(e-\frac{e}{c}\right)\cdot1.
\end{equation}
Equation \eqref{genereindice2} now follows by direct computation.

{\rm (ii)}: let $G\leq{\rm PGU}(3,q)$ stabilize $T$ with $[G:G_T]=2$. Let $e$ be the order of $G_T$.
Since $G_T$ has odd order, $G_T$ has a complement in $G$, say $G=G_T\rtimes C_2$ with $C_2=\langle\beta\rangle$. Without restriction, we assume that $\beta$ stabilizes $P_3$ and interchanges $P_1$ and $P_2$, so that
$$  \beta=\begin{pmatrix} 0 & \gamma & 0 \\ \gamma^{-1} & 0 & 0 \\ 0 & 0 & 1 \end{pmatrix}, $$
with $\gamma^{q+1}=1$. For any $\sigma=\diag(\delta,\epsilon,1)\in G_T$, we have $\beta^{-1}\sigma\beta=\diag(\epsilon,\delta,1)\in G_T$. Hence, the subgroup $A\leq G_T$ of elements of type $\diag(\lambda,1,1)$ has the same order $a\mid(q+1)$ of the subgroup $B=\beta^{-1}A\beta\leq G_T$ made by the elements of type $\diag(1,\lambda,1)$. The order $c\mid(q+1)$ of the subgroup $C\leq G_T$ of elements of type $\diag(\lambda,\lambda,1)$ is a multiple of $a$; note that $C$ is the center of $G$. The number of homologies in $G_T$ is $d=2a+c-3$, and the homologies generate $K=ABC= A\times C$.

Let $\bar p$ be a prime divisor of $|G_T/K|=\bar{p}^v f$ with $\bar{p}\nmid f$.
As in the proof of Theorem \ref{fissatorepuntuale}, the Sylow $\bar p$-subgroup $\bar P$ of $G_T/K$ is cyclic; $\bar P=\langle\alpha K\rangle$, where $\alpha=\diag(\delta,\epsilon,1)\in G_T$ is a $\bar p$-element and $o(\alpha^{\bar p^{v}})=\bar p^{m}$ is maximal in the Sylow $\bar p$-subgroup of $K$. We can also assume $o(\delta)=o(\alpha)=\bar p^{v+m}$ and $\epsilon=\delta^j$ for some $j\in\{2,\ldots,\bar p^{v+m}-1\}$; otherwise, $o(\epsilon)=o(\alpha)$, $\delta=\epsilon^j$, and analogous arguments hold.
The element $\beta^{-1}\alpha\beta=\diag(\delta^j,\delta,1)$ must be contained in $\langle\alpha\rangle$; otherwise, $\langle\alpha,\beta^{-1}\alpha\beta\rangle$ is a subgroup $C_{\bar p^{v+m}}\times C_{\bar p^{v+m}}$ of $G$ containing homologies out of $K$, a contradiction.
Thus, $\beta^{-1}\alpha\beta=\alpha^k$ for some $k\in\{2,\ldots,\bar p^{v+m}-1\}$, that is, $\delta^{k}=\delta^{j}$ and $\delta^{jk}=\delta$. This implies $k=j=\bar p^{v+m}-1$ and $\alpha=\diag(\delta,\delta^{-1},1)$.

Hence, $\langle\alpha\rangle$ has order a divisor of $q+1$; the same argument for all prime divisors $\bar p$ of $|G_T/K|$ proves that $\frac{e}{a}\mid(q+1)$.
Also, no nontrivial power $\alpha^r=\diag(\delta^{r},\delta^{-r},1)\ne id$ of $\alpha$ can be contained in $C$. Hence, $\gcd(o(\alpha),a)=\gcd(o(\alpha),c)$; otherwise, the product $\langle\alpha\rangle \cdot C$ would contain new homologies in $A$ whose order does not divide $a$, a contradiction. Therefore, $\gcd(\frac{e}{ac},\frac{c}{a})=1$.

Now the value of $\Delta$ is computed as in Equation \eqref{boh} and provides Equation \eqref{genereindice2}, that is the thesis.
\end{proof}

\begin{proposition}\label{indice2dispari}
Let $q$ be odd.
\begin{itemize}
\item[(i)] Let $\ell$, $a$, $c$, and $e$ be positive integers satisfying $e\mid(q+1)^2$, $c\mid(q+1)$, $\ell\mid c$, $a\mid c$, $ac\mid e$, $\frac{e}{a}\mid(q+1)$, and $\gcd(\frac{e}{ac},\frac{c}{a})=1$.
If $2\mid a$ or $2\nmid c$, we also require that $2\nmid \frac{e}{ac}$.
Then there exists a subgroup $G\leq (C_{q+1}\times C_{q+1})\rtimes \mathbf{S}_3$ of order $2e$ such that $|G\cap(C_{q+1}\times C_{q+1})|=e$ and 
\begin{equation}\label{genereindice2dispari}
g(\cH_q/G)=\frac{(q+1)\left(q-2a-c+1-h\right)-2k + 4e}{4e},
\end{equation}
where
$$ (h,k)=\begin{cases}
\left(\frac{e}{c},\frac{e}{2}\right) & \qquad \textrm{if} \qquad 2a\nmid (q+1) \,;\\
\left(\frac{e}{c},0\right) & \qquad \textrm{if} \qquad 2a \mid (q+1),\; 2a\nmid c\;; \\
\left(0,e\right) & \qquad \textrm{if} \qquad 2a\mid c,\;2\ell\nmid(q+1)\;; \\
\left(0,0\right) & \qquad \textrm{if} \qquad 2a\mid c,\; 2\ell\mid (q+1),\; 2\ell\nmid c\;; \\
\left(\frac{2e}{c},0\right) & \qquad \textrm{if} \qquad 2a\mid c,\; 2\ell\mid c\;. \\
\end{cases} $$
\item[(ii)] Conversely, if $G\leq(C_{q+1}\times C_{q+1})\rtimes \mathbf{S}_3$ and $G\cap(C_{q+1}\times C_{q+1})$ has index $2$ in $G$, then the genus of $\cH_q/G$ is given by Equation \eqref{genereindice2dispari}, where: $e=|G|/2$; without loss of generality, $a-1$ is the number of homologies in $G$ with center $P_1$ which is equal to the number of homologies in $G$ with center $P_2$, and $c-1$ is the number of homologies in $G$ with center $P_3$; $\ell=\frac{o(\beta)}{2}$ for some $\beta\in G\setminus G_T$; $\ell,a,c,e$ satisfy the numerical assumptions in point {\it (i)}.
\end{itemize}
\end{proposition}

\begin{proof}
{\it (i)}: let $\ell,a,c,e$ be positive integers satisfying the assumptions in point {\it (i)}.
Choose $\lambda,\mu,\rho\in\mathbb{F}_{q^2}$ such that $o(\lambda)=a$, $o(\mu)=c$, and $o(\rho)=e/a$. Define $A=\langle\diag(\lambda,1,1)\rangle$, $C=\langle\diag(\mu,\mu,1)\rangle$, $D=\langle\diag(\rho,\rho^{-1},1)\rangle$, $S=ACD= (A\times C)D$, and
$$
\beta=\begin{pmatrix}
0 & t & 0 \\ 1 & 0 & 0 \\ 0 & 0 & 1
\end{pmatrix},
$$
where $t\in\mathbb{F}_{q^2}$ has order $\ell$. Let $G=\langle S,\beta\rangle$. We have that $G\leq \PGU(3,q)$, $G$ stabilizes $T$, and $G_T=S$. The numerical assumptions imply that, $\beta^2=\diag(t,t,1)\in C$, $\beta$ normalizes $G_T$, and $|G|=2e$.
The homologies in $G_T$ are in $K=A\times C$, and their number is $2a+c-3$; any other nontrivial element in $G_T$ is of type (B1).
Let $\sigma\in G\setminus G_T$, which is uniquely written as $\tau\beta$ with $\tau\in G_T$.
We have $\tau=\zeta\alpha\xi$, where $\zeta\in D$, $\alpha$ is a fixed element of $A$ and $\xi$ is a fixed element of $C$; write $\alpha=\diag(\lambda^r,1,1)$ and $\xi=\diag(\mu^s,\mu^s,1)$, where $r\in\{0,\ldots,a-1\}$ and $s\in\{0,\ldots,c-1\}$ are uniquely determined.
By direct checking, $\sigma^2=\diag(\lambda^r \mu^{2s}t,\lambda^r \mu^{2s}t,1)$. Hence, $\sigma$ is either of type (A), or (B1), or (B2), according to $\lambda^r \mu^{2s}t=1$, or $1\ne o(\lambda^r \mu^{2s}t)\mid\frac{q+1}{2}$, or $o(\lambda^r \mu^{2s}t)\nmid \frac{q+1}{2}$, respectively.
Note that the type of $\sigma$ does not depend on $\zeta$.
Thus, to find the number of elements $\sigma=\tau\beta\in G\setminus G_T$ of a certain type, we must count the number of elements $\tau\in K$ such that $\sigma$ is of that type, and then multiply this number by the index
 $\frac{e}{ac}$ of $K$ in $G_T$; to this aim, we distinguish different cases arising from the numerical conditions on $\ell,a,c$. Recall that $\beta$ is fixed as above.
\begin{itemize}
\item $a\nmid\frac{q+1}{2}$: this implies $c\nmid\frac{q+1}{2}$.
The number of elements of type (B2) in $G\setminus G_T$ is equal to $\frac{a}{2}\cdot c\cdot\frac{e}{ac}=\frac{e}{2}$. In fact, if $\ell\mid\frac{q+1}{2}$, then $o(\lambda^r\mu^{2s}t)\nmid\frac{q+1}{2}$ if and only if $r$ is odd; if $\ell\nmid\frac{q+1}{2}$, then $o(\lambda^r\mu^{2s}t)\nmid\frac{q+1}{2}$ if and only if $r$ is even.

Suppose $\ell\nmid\frac{q+1}{2}$ and $r$ is even, or $\ell\mid\frac{q+1}{2}$ and $r$ is odd. As already observed, $\sigma$ is of type (A) if and only if $\mu^{2s}=(\lambda^r t)^{-1}$. Thus, $G\setminus G_T$ has $\frac{a}{2}\cdot 2\cdot\frac{e}{ac}=\frac{e}{c}$ elements of type (A), and $\frac{a}{2}\cdot (c-2)\cdot\frac{e}{ac}$ elements of type (B1).
\item $a\mid\frac{q+1}{2}$ and $a\nmid\frac{c}{2}$: in this case 
$o(\lambda^r\mu^{2s}t)\mid\frac{q+1}{2}$ and $G\setminus G_T$ has no elements of type (B2).
For any $\ell$ 
there exist exactly $\frac{a}{2}$ values of $r$ such that $o((\lambda^r t)^{-1})\mid \frac{c}{2}$ (as above, we consider the cases $r$ even and $r$ odd separately)
and there are exactly $2$ values of $s$ for which $\lambda^r\mu^{2s}t=1$. In this way we get $\frac{a}{2}\cdot2\cdot\frac{e}{ac}=\frac{e}{c}$ elements of $G\setminus G_T$ are of type (A), the other ones are of type (B1).
\item $a\mid\frac{c}{2}$ and $\ell\nmid\frac{q+1}{2}$: in this case $o(\lambda^r\mu^{2s}t)\nmid\frac{q+1}{2}$ since $o(\lambda^r\mu^{2s})\mid\frac{q+1}{2}$ and $o(t)\nmid\frac{q+1}{2}$; hence, all elements of $G\setminus G_T$ are of type (B2).
\item $a\mid\frac{c}{2}$ and $\ell\mid\frac{q+1}{2}$ and $\ell\nmid\frac{c}{2}$: in this case $o(\lambda^r\mu^{2s}t)\mid\frac{q+1}{2}$ and $o(\lambda^r\mu^{2s})\ne o(t)$, so that $\lambda^r\mu^{2s}t\ne1$. Hence all elements of $G\setminus G_T$ are of type (B1).
\item $a\mid\frac{c}{2}$ and $\ell\mid\frac{c}{2}$: in this case $o(\lambda^r\mu^{2s}t)\mid\frac{q+1}{2}$ and $G\setminus G_T$ contains no elements of type (B2). Also, $o((\lambda^r t)^{-1})\mid o(\mu^2)$ and $G\setminus G_T$ contains exactly $a\cdot2\cdot\frac{e}{ac}=\frac{2e}{c}$ elements of type (A).
\end{itemize}
Denote by $h$ and $k$ the number of elements of $G\setminus G_T$ of type (A) and (B2), respectively.
By direct checking, the Riemann-Hurwitz formula and Theorem \ref{caratteri} provide the value given in Equation \eqref{genereindice2dispari} for the genus of $\cH_q/G$.

{\it (ii)}: let $G\leq\PGU(3,q)$ stabilize $T$ with $[G:G_T]=2$.
Put $e=|G_T|$. For any $\beta\in G\setminus G_T$, $\beta^2\in G_T$, and we can assume without loss of generality that $\beta$ stabilizes $P_3$ and interchanges $P_1$ and $P_2$, so that
$$ \beta=\begin{pmatrix} 0 & \gamma_1 & 0 \\ \gamma_2 & 0 & 0 \\ 0 & 0 & 1 \end{pmatrix}  $$
for some $\gamma_1,\gamma_2$ with $\gamma_1^{q+1}=\gamma_2^{q+1}=1$.
With same arguments as in the proof of Proposition \ref{indice2pari}, one can show that: the number of homologies with center $P_1$, say $a-1$, is equal to the number of homologies with center $P_2$; the number of homologies with center $P_3$, say $c-1$, satisfies $a\mid c$ and $c\mid(q+1)$;
the subgroup of $G_T$ generated by the homologies of $G_T$ is $K=A\times C$, where $A$ and $C$ are given by the identity together with the homologies having center $P_1$ and $P_3$, respectively;
from $K\leq G_T$ follows that
$ac\mid e$; $\frac{e}{a}\mid(q+1)$; and $\gcd(\frac{e}{ac},\frac{c}{a})=1$.
Since there are no homologies in $G_T\setminus K$, same arguments as in Theorem \ref{fissatorepuntuale} give $2\nmid \frac{e}{ac}$ when $2\mid a$ or $2\nmid c$.
The element $\beta^2=\diag(\gamma_1\gamma_2,\gamma_1\gamma_2,1)$ is either trivial or a homology with center $P_3$; hence, $\beta^2\in C$ and $\ell:=o(\gamma_1\gamma_2)=\frac{o(\beta)}{2}$ is a divisor of $c$.
Now the value of $\Delta$ in the Riemann-Hurwitz formula
$$2g(\cH_q)-2=|G|(2g(\cH_q/G)-2)+\Delta$$
can be computed as above and provides Equation \eqref{genereindice2dispari}, so that the claim follows.
\end{proof}

\subsection{The pointwise stabilizer of $T$ has index $3$ in $G$}\label{sec:indice3}

In this section $[G:G_T]=3$. We consider the cases $3\nmid(q+1)$ and $3\mid(q+1)$ separately.

\begin{proposition}\label{indice3caso1}
Let $q$ be such that $3\nmid(q+1)$.
\begin{itemize}
\item[(i)]
Let $a$ and $e$ be positive integers satisfying $e\mid(q+1)^2$, $a^2\mid e$, $\frac{e}{a}\mid(q+1)$, $2\nmid\frac{e}{a^2}$, and $\gcd(\frac{e}{a^2},a)=1$.
We also require that there exists a positive integer $m\leq \frac{e}{a^2}$ such that $\frac{e}{a^2}\mid(m^2-m+1)$.
Then there exists a subgroup $G\leq(C_{q+1}\times C_{q+1})\rtimes \mathbf{S}_3$ of order $3e$ such that $|G\cap(C_{q+1}\times C_{q+1})|=e$ and
\begin{equation}\label{genereindice3caso1}
g(\cH_q/G)=\frac{(q+1)(q-3a+1)+2e}{6e}.
\end{equation}
\item[(ii)] Conversely, if $G\leq(C_{q+1}\times C_{q+1})\rtimes \mathbf{S}_3$ and $G\cap(C_{q+1}\times C_{q+1})$ has index $3$ in $G$, then the genus of $\cH_q/G$ is given by Equation \eqref{genereindice3caso1}, where: $e=|G|/3$; the number of homologies in $G$ with center $P_i$ is $a-1$ for $i=1,2,3$; there exists $m$ such that $a,e,m$ satisfy the numerical assumptions in point {\it (i)}.
\end{itemize}
\end{proposition}

\begin{proof}
{\it (i)}:
let $a$, $e$, and $m$ satisfy the assumptions in point {\it (i)}.
Define the following elements and subgroups of ${\rm \PGU}(3,q)$: $K=\{\diag(\lambda,\mu,1)\mid \lambda^a=\mu^a=1\}$, $\alpha=\diag(\rho,\rho^m,1)$ where $\rho\in\mathbb{F}_{q^2}$ is an element of order $\frac{e}{a^2}$, $S=\langle K,\alpha\rangle=K\times\langle\alpha\rangle$ as $\gcd\left(\frac{e}{a^2},a\right)=1$,
$$ \beta=\begin{pmatrix} 0 & 1 & 0 \\ 0 & 0 & 1 \\ 1 & 0 & 0 \end{pmatrix}, $$
and $G=\langle S,\beta \rangle\leq\PGU(3,q)$. Note that $S$ stabilizes $T$ pointwise, $\beta$ has order $3$, acts on $T$, and normalizes $K$.
Also, $\beta^{-1}\alpha\beta=\alpha^{-m}$, so that $\beta$ normalizes $S$.
Thus, $G_T=S$ has order $e$, and $G=S\rtimes\langle\beta\rangle$ has order $3e$.

We show that $g(\cH_q/G)$ is given by Equation \eqref{genereindice3caso1}.
The subgroup $K$ contains exactly $3(a-1)$ elements of type (A); the other nontrivial elements of $K$ are of type (B1).
Suppose that $G_T\ne K$, that is, $\alpha$ is not trivial. Let $\sigma\in G_T\setminus K$ and write $\sigma=\diag(\lambda\rho^{j},\mu\rho^{jm},1)$ with $\lambda^a=\mu^a=1$ and $0<j<\frac{e}{a^2}$. The element $\sigma$ is of type (A) if and only if either $\lambda\rho^j=1$, or $\mu\rho^{jm}=1$, or $\lambda\rho^j=\mu\rho^{jm}$. By direct checking using the numerical conditions on the parameters, none of these conditions holds.
Therefore, every element in $G_T\setminus K$ is of type (B1).

Let $\sigma\in G\setminus G_T$, so that $\sigma=\diag(\delta,\epsilon,1)\cdot\beta$ or $\sigma=\diag(\delta,\epsilon,1)\cdot\beta^2$ for some $(q+1)$-th roots of unity $\delta$ and $\epsilon$. By direct checking $\sigma$ has order $3$. If $3\mid(q-1)$, then $\sigma$ is of type (B2).
If $3\mid q$, then $\sigma$ is of type (D), because no element of type (C) can act transitively on the vertices of $T$ (as it stabilizes each line through its center).
Equation \eqref{genereindice3caso1} now follows from Theorem \ref{caratteri} and the Riemann-Hurwitz formula. In fact $G$ has $3(a-1)$ elements of type (A), $e-3a+2$ elements of type (B1), and $2e$ elements of type either (B2) or (D) according to $3\mid(q-1)$ or $3\mid q$, respectively.

{\it (ii)}: let $G\leq\PGU(3,q)$ stabilize $T$ with $[G:G_T]=3$, $e$ be the order of $G_T$, and $\beta$ be an element of $G\setminus G_T$.
Then $\beta$ acts transitively on $T$ and we can assume that
$$ \beta=\begin{pmatrix} 0 & \gamma_1 & 0 \\ 0 & 0 & \gamma_2 \\ 1 & 0 & 0 \end{pmatrix} $$
for some $(q+1)$-th roots of unity $\gamma_1,\gamma_2$.
The element $\beta$ has order $3$ and normalizes $G_T$, so that $G=G_T\rtimes \langle\beta\rangle$.
Since $\beta$ is a $3$-cycle on $T$,
the number of elements of type (A) with center $P_i$ is the same for $i=1,2,3$, say $a-1$; clearly, $a$ divides $q+1$.
Then the elements of type (A) in $G_T$ generate the subgroup $K=\{\diag(\lambda,\mu,1)\mid\lambda^a=\mu^a=1\}$ of order $a^2$; this implies $a^2\mid e$.
From the proof of Theorem \ref{fissatorepuntuale} applied to $G_T$ follows that $\frac{e}{a}\mid(q+1)$, $\frac{e}{a^2}$ is odd,
and
$G_T/K$ is cyclic; say $G_T/K=\langle\alpha K\rangle$.

Suppose that $|G_T/K|=\frac{e}{a^2}\ne1$, that is, $\alpha\notin K$ and $\alpha$ is of type (B1). We can assume that $\alpha=\diag(\rho,\rho^m,1)$ with $o(\alpha)=o(\rho)=\frac{e}{a^2}$ and $1<m<\frac{e}{a^2}$.
If $d=\gcd(o(\rho),|K|)$, then $\alpha^d\in K$ because $K$ contains all elements of order $d$ in $C_{q+1}\times C_{q+1}$. Thus, we can replace $\alpha$ with $\alpha^d$ and assume that $\gcd(o(\alpha),|K|)=1$; hence, $\gcd(\frac{e}{a^2},a)=1$ and $G_T=K\times\langle\alpha\rangle$.
Consider $\tilde\alpha=\beta^{-1}\alpha\beta=\diag(\rho^{-m},\rho^{1-m},1)\in G_T$. If $\tilde\alpha\notin\langle\alpha\rangle$, then as in the proof of Theorem \ref{fissatorepuntuale} we have that $\langle\alpha,\tilde\alpha\rangle\cong C_{\frac{e}{a^2}}\times C_{\frac{e}{a^2}}$ contains elements of type (A) not in $K$, a contradiction. Hence, $\tilde \alpha=\alpha^j$ for some $j$ with $1<j<\frac{e}{a^2}$. By direct computation, this is equivalent to require that $j=\frac{e}{a^2}-m$ and $\frac{e}{a^2}\mid(m^2-m+1)$. The same condition is required for $(\beta^2)^{-1}\alpha\beta^2\in\langle\alpha\rangle$.
Now the same argument as in point {\it (i)} provides the type of the elements of $G$ and therefore the genus of $\cH_q/G$ by means of Theorem \ref{caratteri} and the Riemann-Hurwitz formula.
\end{proof}

\begin{proposition}\label{indice3caso2}
Let $q$ be such that $3\mid(q+1)$.
\begin{itemize}
\item[(i)]
Let $a$, $e$, and $\ell$ be positive integers satisfying $e\mid(q+1)^2$, $a^2\mid e$, $\frac{e}{a}\mid(q+1)$, $2\nmid\frac{e}{a^2}$, $\gcd(\frac{e}{a^2},a)=1$, and $\ell\mid(q+1)$.
We also require that there exists a positive integer $m\leq \frac{e}{a^2}$ such that $\frac{e}{a^2}\mid(m^2-m+1)$.
Then there exists a subgroup $G\leq(C_{q+1}\times C_{q+1})\rtimes \mathbf{S}_3$ of order $3e$ such that $|G\cap(C_{q+1}\times C_{q+1})|=e$ and
\begin{equation}\label{genereindice3caso2}
g(\cH_q/G)=\frac{(q+1)(q-3a+1)+h\cdot e}{6e}, \qquad\textrm{with}\qquad h=\begin{cases}
2 & \textrm{if}\quad a\nmid\frac{q+1}{3}, \\
0 & \textrm{if}\quad a\mid\frac{q+1}{3},\;\ell\nmid\frac{q+1}{3}, \\
6 & \textrm{if}\quad a\mid\frac{q+1}{3},\;\ell\mid\frac{q+1}{3}. \\
\end{cases}
\end{equation}
\item[(ii)] Conversely, if $G\leq(C_{q+1}\times C_{q+1})\rtimes \mathbf{S}_3$ and $G\cap(C_{q+1}\times C_{q+1})$ has index $3$ in $G$, then the genus of $\cH_q/G$ is given by Equation \eqref{genereindice3caso2}, where: $e=|G|/3$; the number of homologies in $G$ with center $P_i$ is $a-1$ for $i=1,2,3$; there exist $\ell$ and $m$ such that $a,e,\ell,m$ satisfy the numerical assumptions in point {\it (i)}.
\end{itemize}
\end{proposition}

\begin{proof}
{\it (i)}:
let $a,e,\ell,m$ satisfy the assumptions in point {\it (i)}. Define $K,\alpha,S$ as in the proof of Proposition \ref{indice3caso1} and $\beta\in\PGU(3,q)$ by
$$ \beta=\begin{pmatrix} 0 & t & 0 \\ 0 & 0 & 1 \\ 1 & 0 & 0 \end{pmatrix}, $$
where $t\in\mathbb{F}_{q^2}$ has order $\ell$.
Let $G=\langle S,\beta\rangle$. Arguing as in the proof of Proposition \ref{indice3caso1} we have that $S=G_T$ has order $e$, $G=G_T\rtimes\langle\beta\rangle$ has order $3e$, $G_T$ has $3(a-1)$ elements of type (A) and any other nontrivial element of $G_T$ is of type (B1). Also, every element $\sigma\in G\setminus G_T$ has order $3$ and can be uniquely written as $\sigma=\kappa\alpha^j\beta$ or $\sigma=\kappa\alpha^j\beta^2$, where $\kappa=\diag(\lambda,\mu,1)\in K$ and $\alpha^j=\diag(\rho^j,\rho^{jm},1)$.
The element $\sigma$ has exactly $3$ fixed points in $\PG(2,\bar{\mathbb{F}}_{q^2})$, namely $(\bar x,(t\lambda\rho^{j})^{-1}\bar{x}^2,1)$ where $\bar{x}^3=t\lambda\mu\rho^{j(m+1)}$.
If the fixed points of $\sigma$ are $\mathbb{F}_{q^2}$-rational then $\sigma$ is of type (B1), otherwise $\sigma$ is of type (B3).
Then $\sigma$ is either of type (B1) or of type (B3), according to $o(t\lambda\mu\rho^{j(m+1)})\mid\frac{q+1}{3}$ or $o(t\lambda\mu\rho^{j(m+1)})\nmid\frac{q+1}{3}$, respectively.
If $3\mid o(\rho)=\frac{e}{a^2}$, then $3\mid(m^2-m+1)$ and hence $3\mid(m+1)$, so that $o(\rho^{j(m+1)})\mid\frac{q+1}{3}$; if $3\nmid o(\rho)$, then again $o(\rho^{j(m+1)})\mid\frac{q+1}{3}$. In any case, $\sigma$ is of type (B1) or (B3) according to $o(t\lambda\mu)\mid\frac{q+1}{3}$ or $o(t\lambda\mu)\nmid\frac{q+1}{3}$, respectively.
We make clear now the type of $\sigma$ in relation to the assumptions on $a$ and $\ell$.
\begin{itemize}
\item If $a\mid\frac{q+1}{3}$ and $\ell\mid\frac{q+1}{3}$, then the $2e$ elements of $G\setminus G_T$ are all of type (B1).
\item If $a\mid\frac{q+1}{3}$ and $\ell\nmid\frac{q+1}{3}$, then the $2e$ elements of $G\setminus G_T$ are all of type (B3).
\item If $a\nmid\frac{q+1}{3}$, then $G\setminus G_T$ contains $\frac{a^2}{3}\cdot[G_T:K]\cdot 2=\frac{2e}{3}$ elements of type (B1) and $2e-\frac{2e}{3}=\frac{4e}{3}$ elements of type (B3).
This number can be obtained counting the elements of type (B1) as follows.
Consider a primitive $(q+1)$-th root of unity $\epsilon$ and write $t=\epsilon^r$, $\lambda\mu=\epsilon^s$. Then $\sigma$ is of type (B1)
 if and only if $s\equiv-r\pmod3$. When $s$ is given, we have $\frac{a}{3}$ choices for $s$ such that $s\equiv-r\pmod3$. Hence, we have $\frac{a^2}{3}$ choices for the couple $(\lambda,\mu)$ such that $\sigma$ is of type (B1).
\end{itemize}
Then Equation \eqref{genereindice3caso2} follows by Theorem \ref{caratteri} and the Riemann-Hurwitz formula.

{\it (ii)}: let $G\leq\PGU(3,q)$ stabilize $T$ with $[G:G_T]=3$ and $|G_T|=e$. Then we can argue as in the proof of Proposition \ref{indice3caso1} to prove that $G=G_T\rtimes\langle\beta\rangle$ with
$$ \beta=\begin{pmatrix} 0 & \gamma_1 & 0 \\ 0 & 0 & \gamma_2 \\ 1 & 0 & 0 \end{pmatrix} $$
and $\ell:=o(\gamma_1\gamma_2)\mid(q+1)$, $G_T=K\times\langle\alpha\rangle$ where $K$ is the subgroup generated by the elements of type (A), $|K|=a^2$, $\alpha=\diag(\rho,\rho^{m},1)$, and the parameters $e,a,m,\ell$ satisfy the assumptions in point {\it (i)}.
Now the genus of $\cH_q/G$ is obtained as in point {\it (i)} by means of Theorem \ref{caratteri} and the Riemann-Hurwitz formula.
\end{proof}

\subsection{The pointwise stabilizer of $T$ has index $6$ in $G$}\label{sec:indice6}

In this section $[G:G_T]=6$.

\begin{proposition}\label{indice6}
\begin{itemize}
\item[(i)] Let $a$ be a divisor of $q+1$.
We choose $e=a^2$ if $3\nmid(q+1)$ or $3\mid a$; $e\in\{a^2,3a^2\}$ if $3\mid(q+1)$ and $3\nmid a$.
Then there exists a subgroup $G\leq(C_{q+1}\times C_{q+1})\rtimes \mathbf{S}_3$ of order $6e$ such that $|G\cap(C_{q+1}\times C_{q+1})|=e$ and
\begin{equation}\label{genereindice6}
g(\cH_q/G)=\frac{(q+1)(q-3a+1-\frac{3e}{a})-2r-3s-t+12e}{12e},
\end{equation}
where
$$
r=\begin{cases} \frac{7e}{2} & \textrm{if}\quad q\equiv0\,\textrm{ or }\,1\!\!\!\pmod3,\;\textrm{$q$ is odd},\;a\nmid\frac{q+1}{2}, \\
\frac{3e}{2} & \textrm{if}\quad q\equiv2\!\!\!\pmod3,\;\textrm{$q$ is odd},\;a\nmid\frac{q+1}{2}, \\
2e & \textrm{if}\quad q\equiv0\,\textrm{ or }\,1\!\!\!\pmod3,\;\textrm{either $q$ is even or}\;a\mid\frac{q+1}{2}, \\
0 & \textrm{if}\quad q\equiv2\!\!\!\pmod3,\;\textrm{either $q$ is even or}\;a\mid\frac{q+1}{2}; \end{cases}
 $$
$$
s= \begin{cases} \frac{4e}{3} & \textrm{if}\quad q\equiv2\!\!\!\pmod3\;\textrm{ and }\; a\nmid\frac{q+1}{3}, \\
0 & \textrm{otherwise}; \end{cases}
\quad
t=\begin{cases} 0 & \textrm{if $q$ is odd,}\quad \\ 3e & \textrm{if $q$ is even.} \end{cases}
$$
\item[(ii)] Conversely, if $G\leq(C_{q+1}\times C_{q+1})\rtimes \mathbf{S}_3$ and $G\cap(C_{q+1}\times C_{q+1})$ has index $6$ in $G$, then the genus of $\cH_q/G$ is given by Equation \eqref{genereindice6}, where: $e=|G|/6$; the number of homologies in $G$ with center $P_i$ is $a-1$ for $i=1,2,3$; $a$ and $e$ satisfy the numerical assumptions in point {\it (i)}.
\end{itemize}
\end{proposition}

\begin{proof}
{\it (i)}: let $a$ and $e$ satisfy the numerical assumptions in point {\it (i)}.
In $\PGU(3,q)$ define
$$ K=\left\{\begin{pmatrix} \lambda & 0 & 0 \\ 0 & \mu & 0 \\ 0 &  0 & 1 \end{pmatrix}:\lambda^a=\mu^a=1\right\}, \; \varphi=\begin{pmatrix} 0 & 1 & 0 \\ 0 & 0 & 1 \\ 1 & 0 & 0 \end{pmatrix}, \; \psi=\begin{pmatrix} 0 & 1 & 0 \\ 1 & 0 & 0 \\ 0 & 0 & 1 \end{pmatrix},\; C=\langle\varphi,\psi\rangle \cong S_3.  $$
If $e=3a^2$, define $\alpha=\diag(\rho,\rho^{-1},1)\in\PGU(3,q)$ with $\rho^3=1$ and $S=\langle K,\alpha\rangle\cong K\times\langle\rho\rangle$; if $e=a^2$, define $S=K$.
Let $G=\langle S,C\rangle$.
By direct checking, $\varphi$ and $\psi$ normalize $K$ and $\alpha$, so that $G$ is a semmidirect product $S\rtimes C$.
Also, $G_T=S$ and $G_T$ has index $6$ in $G$.

As usual, we count the elements of different type in $G$ as described in Lemma \ref{classificazione}.
The subgroup $G_T$ contains exactly $3(a-1)$ elements of type (A); any other nontrivial element of $G_T$ is of type (B1).
The elements in $G\setminus G_T$ are contained in subgroups $L$ such that either $[L:G_T]=2$ or $[L:G_T]=3$; here, $L\cap(C_{q+1}\times C_{q+1})=G_T$.
Thus, we can apply either Propositions \ref{indice2pari} and \ref{indice2dispari}, or Propositions \ref{indice3caso1} and \ref{indice3caso2}.
Equation \eqref{genereindice6} then follows by the Riemann-Hurwitz formula and Theorem \ref{caratteri}.

{\it (ii)}: let $G\leq\PGU(3,q)$ stabilize $T$ with $[G:G_T]=6$ and $|G_T|=e\mid(q+1)^2$.
The factor group $G/G_T$ acts faithfully on $T$, hence $G/G_T\cong \mathbf{S}_3$. Let $G/G_T=\langle\varphi G_T\rangle\rtimes\langle\psi G_T\rangle$ where $\varphi,\psi\in G\setminus G_T$ satisfy $\varphi^3\in G_T$ and $\psi^2\in G_T$.
We can assume that
$$
\varphi=\begin{pmatrix} 0 & \delta_1 & 0 \\ 0 & 0 & \delta_2 \\ 1 & 0 & 0  \end{pmatrix}\quad \textrm{and}\quad
\psi=\begin{pmatrix} 0 & \gamma_1 & 0 \\ \gamma_2 & 0 & 0 \\ 0 & 0 & 1  \end{pmatrix}$$
for some $(q+1)$-th roots of unity $\delta_1,\delta_2,\gamma_1,\gamma_2$.
As $(\psi^{-1}\varphi\psi) G_T=\varphi^{-1}G_T$, we have that $\gamma:=\gamma_1\gamma_2$ and $\delta:=\frac{\delta_1^2\gamma_2}{\delta_2\gamma_1^2}$ are $a$-th roots of unity.
Then $\diag(1,\gamma^{-1},1)\in G_T$;
hence, we can replace $\psi$ with $\diag(1,\gamma^{-1},1)\cdot\psi$ and assume that $\gamma_2=\gamma_1^{-1}$.
and $\delta_1\delta_2=\frac{\delta_1^3}{\gamma_1^3\delta}$.
This yields $C=\langle\varphi,\psi\rangle\cong S_3$ is a complement for $G_T$ and $G=G_T\rtimes C$.

From $\varphi^{-1}G_T\varphi=G_T$ follows as in the proof of Proposition \ref{indice3caso1} that the elements of type (A) in $G_T$ with center $P_1,P_2,P_3$ are in the same number $a-1$, $a\mid(q+1)$, and generate a subgroup $K\leq G_T$ of order $|K|=a^2$.
Suppose that there exists $\alpha\in G_T\setminus K$, that is, $\frac{e}{a^2}>1$.
From $\psi^{-1}G_T\psi=G_T$ follows as in the proof of Proposition \ref{indice2pari} that $\alpha=\diag(\lambda,\lambda^{-1},1)$ for some $\lambda$ with $\lambda^e=1$.
From $\varphi^{-1}G_T\varphi=G_T$ follows as in the proof of Proposition \ref{indice3caso1} that $\frac{e}{a^2}\mid(\ell^2-\ell+1)$ with $\ell=-1$. Hence $o(\alpha)=3$, $e=3a^2$, and $3\nmid a$, since there are no elements of order $3$ in $(C_{q+1}\times C_{q+1})\setminus K$ if $3\mid a$.
Since now the structure of $G$ has been determined, the genus of $\cH_q/G$ can be computed arguing as above.
Here, use the following fact: if $3\mid(q+1)$ and $a\mid\frac{q+1}{3}$, then $o(\delta_1\delta_2)\mid\frac{q+1}{3}$; if $3\mid(q+1)$ and $a\nmid\frac{q+1}{3}$, then $o(\delta_1\delta_2)\nmid\frac{q+1}{3}$.
\end{proof}

\section{$G$ stabilizes a point $P\in\PG(2,q^2)\setminus\cH_q$ with $q$ even}\label{sec:polopolare}

In this section $q$ is even and $G$ is supposed to be contained in the maximal subgroup $M$ of $\PGU(3,q)$ of order $|M|=q(q-1)(q+1)^2$ stabilizing a pole-polar pair $(P,\ell)$ with respect to the unitary polarity associated to $\cH_q(\mathbb{F}_{q^2})$; here $P\in\PG(2,q^2)\setminus\cH_q$ and $|\ell\cap\cH_q|=|\ell\cap\cH_q(\mathbb{F}_{q^2})|=q+1$.

Following \cite[Section 3]{CKT2}, we use the plane model \eqref{M3} of $\cH_q$ and assume up to conjugation in $\PGU(3,q)$ that $P=(0,0,1)$ and $\ell$ is the line at infinity $Z=0$.
Note that the $\mathbb{F}_{q}$-rational points of $\cH_q$ are exactly the $q+1$ points of $\ell\cap\cH_q$.
Then
$$ M=\left\{\begin{pmatrix} a & b & 0 \\ c & d & 0 \\ 0 & 0 & 1 \end{pmatrix} \mid a,b,c,d\in\mathbb{F}_{q^2},\, ac^q-a^qc=0,\,bd^q-b^qd=0,\,bc^q-a^qd=-1,\,ad^q-b^qc=1 \right\}. $$
If $\sigma\in M$, we denote by $\det(\sigma)$ the determinant of the representative of $\sigma$ with entry $1$ on the third row and column; see Remark \ref{abuso}. Let
$$ H=\left\{\begin{pmatrix} a & b & 0 \\ c & d & 0 \\ 0 & 0 & 1 \end{pmatrix}\in M \mid a,b,c,d\in\mathbb{F}_{q},\, ad-bc=1 \right\}\leq M,\quad \Omega=\left\{\begin{pmatrix} \lambda & 0 & 0 \\ 0 & \lambda & 0 \\ 0 & 0 & 1 \end{pmatrix}\right\}\leq M. $$
The group $H$ is isomorphic to $\SL(2,q)$, and its action on $\ell\cap\cH_q$ is equivalent to the action of $\SL(2,q)$ in its usual permutation representation on $\mathbb{F}_{q}$.
The group $\Omega$ is cyclic of order $q+1$ and made by the homologies of $\PGU(3,q)$ with center $P$.
As $q$ is even, $H \cap \Omega$ is trivial and hence
$$M=H\times\Omega.$$
Let $G\leq M$, $G(H)=G\cap H$, and $G(\Omega)=G\cap \Omega$. Then $G/G(\Omega)$ acts faithfully on $\ell\cap\cH_q$ and is isomorphic to a subgroup of $\PSL(2,q)$.
Throughout this section, we will denote by $\omega\mid(q+1)$ the order of $G(\Omega)$.

\begin{remark}\label{spezzabile}
Let $G\leq M$ be such that $G/G(\Omega)$ is generated by elements whose order is coprime to $q+1$. Then $G=G(H)\times G(\Omega)$.
\end{remark}

\begin{proof}
Let $\alpha_1 G(\Omega),\ldots,\alpha_r G(\Omega)$ be generators of $G/G(\Omega)$ of order $o_1,\ldots,o_r$, respectively. Then $\beta_1=\alpha_1^{q+1},\ldots,\beta_r=\alpha_r^{q+1}$ have the same orders $o_1,\ldots,o_r$; also, $\det(\beta_1)=\cdots=\det(\beta_r)=1$ and hence $\beta_i\in H$ for all $i$'s.
The subgroup $L=\langle\beta_1,\ldots,\beta_r\rangle$ of $G$ induces the whole group $\langle\alpha_1\ldots,\alpha_r\rangle=G/G(\Omega)$ and $L\cap G(\Omega)=\{id\}$, so that $L\cong G/G(\Omega)$. Thus, $L=G(H)$ and claim follows.
\end{proof}

We now compute the genus of $\cH_q/G$ for any $G\leq M$, using Theorem \ref{Di} for $G/G(\Omega)$; recall that $\SL(2,q)\cong\PGL(2,q)\cong\PSL(2,q)$ since $q$ is even.

If $G=G(\Omega)$, then $g(\cH_q/G)$ is computed in Theorem \ref{fissatorepuntuale} (see also \cite[Theorem 5.8]{GSX}) as $g(\cH_q/G)=1+\frac{(q+1)(q-|G|-1)}{2|G|}$; hence, in the following we will assume that $G/G(\Omega)$ is not trivial.

Let $G/G(\Omega)$ be cyclic of order a divisor of $q+1$, say $G/G(\Omega)=\langle\alpha G(\Omega)\rangle$ with $\alpha^{q+1}\in G(\Omega)$.
From Lemma \ref{classificazione}, $\alpha$ is either of type (A) or of type (B1).
If $\alpha$ is of type (A), then the center of $\alpha$ is a point of $\ell$, because $\alpha\notin G(\Omega)$ and $\alpha$ commutes with $G(\Omega)$ elementwise. In any case, $\alpha$ stabilizes pointwise a self-polar triangle $\{P_0,P_1,P_2\}\subset\PG(2,q^2)\setminus\cH_q$.
For any $\beta\in G$ we have $\beta=\alpha^d \bar g$ for some integer $d$ and some $\bar g\in G(\Omega)$; hence, $\beta$ stabilizes $\{P_0,P_1,P_2\}$ pointwise.
Therefore, the groups $G\leq M$ such that $|G/G(\Omega)|\mid(q+1)$ are exactly the groups considered in Theorem \ref{fissatorepuntuale}, and the genus of $\cH_q/G$ is characterized by Equation \eqref{generefissatore}.

\begin{proposition}
Let $G\leq M$ be such that $G/G(\Omega)\cong\PSL(2,2)$. If $n$ is even, then
$$ g(\cH_q/G)=
\frac{q^2-\omega q-3q+4\omega-4}{12\omega}. $$
If $n$ is odd, either
\begin{equation}\label{primocasocheaccade}
g(\cH_q/G)=
\frac{(q+1)(q-\omega-8)+9\omega}{12\omega},
\end{equation}
or
\begin{equation}\label{secondocasocheaccade}
g(\cH_q/G)=
\frac{(q+1)(q-2\cdot 3^k-\omega-2)+9\omega}{12\omega}
\end{equation}
where $3^k$ is the maximal power of $3$ dividing $\omega$ and $k\geq1$.
Both cases \eqref{primocasocheaccade} and \eqref{secondocasocheaccade} occur for some $G\leq M$ with $G/G(\Omega)\cong\PSL(2,2)$.
\end{proposition}

\begin{proof}
Suppose that $3\mid(q-1)$, i.e. $q$ is an even power of $2$.
Then $G=G(H)\times G(\Omega)$ with $G(H)\cong\PSL(2,2)\cong\mathbf{S}_3$.
By Lemma \ref{classificazione}, the nontrivial elements of $G$ are as follows:
$2\omega$ elements of order $3$ times a divisor of $\omega$, of type (B2); $3$ involutions, of type (C); $\omega-1$ elements in $G(\Omega)$, of type (A); $3(\omega-1)$ elements of order $2$ times a nontrivial divisor of $\omega$, of type (E).
The claim follows from the Riemann-Hurwitz formula and Theorem \ref{caratteri}.

Suppose that $3\mid(q+1)$, i.e. $q$ is an odd power of $2$.
The group $G/G(\Omega)\cong\PSL(2,2)\cong\mathbf{S}_3$ contains a cyclic normal subgroup $\langle\alpha G(\Omega)\rangle$ of order $3$.
Since $G(\Omega)$ is central in $G$, $\langle\alpha\rangle$ is normal in $G$.
As $\alpha^3\in G(\Omega)$ and $\alpha$ fixes $P$, Lemma \ref{classificazione} implies that $o(\alpha)\mid(q+1)$ and $\alpha$ is of type (A) or (B1).
We show that $\alpha$ is of type (B1). Suppose by contradiction that $\alpha$ is of type (A). As $\alpha\notin \Omega$, the axis of $\alpha$ passes through $P$ and the center of $\alpha$ lies on $\ell$. Hence, $\alpha$ has exactly $2$ fixed points $Q$ and $R$ on $\ell$, and $Q,R\notin \cH_q$; we can assume that $PQ$ is the axis and $R$ is the center of $\alpha$. Let $\beta$ be an involution of $G$ (for instance, choose $\beta=\gamma^{q+1}$ for any involution $\gamma G(\Omega)$ of $G/G(\Omega)$). Then $\beta(Q)=R$ and $\beta(R)=Q$, so that $\beta^{-1}\alpha\beta$ has type (A) with axis $PR$ and center $Q$; this is a contradiction, because $\beta$ normalizes $\alpha$ and $\alpha$ does not contain elements of type (A) with center different from $R$. Thus, $\alpha$ is of type (B1). Hence, $G$ acts on the self-polar triangle $T\subset\PG(2,q^2)\setminus\cH_q$ fixed pointwise by $\alpha$, and the pointwise stabilizer of $T$ in $G$ has index $2$ in $G$. Then $g(\cH_q/G)$ can be computed by means of Proposition \ref{indice2pari}.

If $G=G(H)\times G(\Omega)$, then we apply Proposition \ref{indice2pari} {\it (ii)}, where $a=3$, $c=\omega$, $e=3\omega$. In fact, $c=|G(\Omega)|$, and $a=3$ because the elements of type (A) with axis passing through $P$ are obtained as the product of an element of order $3$ in $G(H)$ by an element of order $3$ in $G(\Omega)$.

If $G\ne G(H)\times G(\Omega)$, then $3\mid|G(\Omega)|$ and $G=(C_{3^{k+1}}\rtimes C_2)\times C_{\omega/3^k}$, where: $3^k$ is the maximal power of $3$ which divides $\omega$; $C_{3^k}$ is generated by an element of type (B1) and order $3^k$ whose cube lies in $G(\Omega)$; $C_2$ is generated by any involution of $G$; $C_{\omega/3^k}$ is the subgroup of $G(\Omega)$ of order $\frac{\omega}{3^k}$.
Such a group $G$ actually exists in $M$. To see this fact, let $\cH_q$ have equation \eqref{M1} and assume $P=(0,0,1)$, $\ell:Z=0$; define in $M$ the elements
$$\alpha=\begin{pmatrix} \lambda\mu & 0 & 0 \\ 0 & \lambda^{-1}\mu & 0 \\ 0 & 0 & 1 \end{pmatrix}, \quad
\beta=\begin{pmatrix} 0 & \rho & 0 \\ \rho^{-1} & 0 & 0 \\ 0 & 0 & 1 \end{pmatrix}, \quad
\delta=\begin{pmatrix} \xi & 0 & 0 \\ 0 & \xi & 0 \\ 0 & 0 & 1 \end{pmatrix},
$$
with $o(\lambda)=3$, $o(\mu)=3^{k+1}$, $o(\rho)\mid(q+1)$, $o(\xi)=\frac{\omega}{3^k}$.
Then $G=\langle\alpha,\beta,\delta\rangle\cong (C_{3^{k+1}}\rtimes C_2)\times C_{\omega/3^k}$ is the desired group.
We apply Proposition \ref{indice2pari} {\it (ii)}, where $a=3^k$, $c=\omega$, $e=3\omega$.
\end{proof}

Apart from the cases considered above of $G/G(\Omega)$ cyclic of order dividing $q+1$ and $G/G(\Omega)\cong\PSL(2,2)$, we will find that the subgroup $G/G(\Omega)$ of $\PSL(2,q)$ is generated by elements of order coprime to $q+1$.
Hence, in the proofs of Propositions \ref{ciclicoq-1} to \ref{conPSL} we make use of Remark \ref{spezzabile} to get $G=G(H)\times G(\Omega)$.

\begin{proposition}\label{ciclicoq-1}
Let $G\leq M$ be such that $G/G(\Omega)$ is cyclic of order $d\mid(q-1)$. Then
$$ g(\cH_q/G)=\frac{(q+1)(q-\omega-1)+2\omega}{2d\omega}. $$
\end{proposition}

\begin{proof}
We have $G=G(H) \times G(\Omega)$ with $G(H)$ cyclic of order $d$. From Theorem \ref{caratteri}, $G$ has $\omega-1$ elements of type (A) and $(d-1)\omega$ elements of type (B2). By the Riemann-Hurwitz formula, $(q+1)(q-2)=d\omega(2g(\cH_q/G)-2)+\Delta$ with $\Delta=(q+1)(\omega-1)+2(d-1)\omega$.
\end{proof}

\begin{proposition}
Let $G\leq M$ be such that $G/G(\Omega)$ is elementary abelian of order $2^f$, $f\leq n$. Then
$$ g(\cH_q/G)=\frac{(q+1)(q-\omega-2^f)+\omega(2^f+1)}{2^{f+1}\omega}. $$
\end{proposition}

\begin{proof}
We have $G=G(H)\times G(\Omega)$ with $G(H)$ elementary abelian of order $2^f$. A nontrivial element $\sigma\in G$ is of type (C) if $o(\sigma)=2$, of type (A) if $o(\sigma)\mid(q+1)$, and of type (E) if $o(\sigma)$ is $2$ times a nontrivial divisor of $q+1$. The claim follows from the Riemann-Hurwitz formula and Theorem \ref{caratteri}.
\end{proof}

Let $G\leq M$ be such that $G/G(\Omega)$ is dihedral of order $2d$ with $d\mid(q+1)$. We have $G=G(H)\times G(\Omega)$, where $G(H)=\langle\alpha\rangle\rtimes\langle\beta\rangle$ is dihedral of order $2d$, $o(\alpha)=d$, $o(\beta)=2$. Since $\beta^{-1}\alpha\beta=\alpha^{-1}$, we have that $\alpha$ is of type (B1) and stabilizes pointwise a self-polar triangle $\{P_0,P_1,P_2\}$, while $\beta$ is of type (C), fixes $P_0$, and interchanges $P_1$ and $P_2$.
Therefore $G$ is already considered in Proposition \ref{indice2pari} and the genus of $\cH_q/G$ is given by Equation \eqref{genereindice2}, with $a=1$, $c=\omega$, $e=d\omega$.

\begin{proposition}
Let $G\leq M$ be such that $G/G(\Omega)$ is dihedral of order $2d$ with $d\mid(q-1)$. Then
$$ g(\cH_q/G)=
\frac{q^2-q\omega-qd+\omega d+\omega-d-1}{4d\omega}. $$
\end{proposition}

\begin{proof}
We have $G=G(H)\times G(\Omega)$ with $G(H)$ dihedral of order $2d$. From Lemma \ref{classificazione}, the nontrivial elements of $G$ are as follows: $(d-1)\omega$ elements of order a divisor of $q^2-1$ but not of $q+1$, of type (B2); $\omega-1$ elements in $G(\Omega)$ of order a divisor of $q+1$, of type (A); $d$ involutions, of type (C); $d(\omega-1)$ elements of order $2$ times a nontrivial divisor of $q+1$, of type (E). Then the claim follows from the Riemann-Hurwitz formula and Theorem \ref{caratteri}.
\end{proof}

\begin{proposition}
Let $G\leq M$ be such that $G/G(\Omega)\cong \mathbf{A}_4$, with $n$ even. Then
$$ g(\cH_q/G)=
\frac{q^2-q\omega+4\omega-3q-4}{24\omega}. $$
\end{proposition}

\begin{proof}
We have $G=G(H)\times G(\Omega)$ with $G(H)\leq H$, $G(H)\cong \mathbf{A}_4$. By Lemma \ref{classificazione}, $G$ contains $3$ elements of type (C), $8\omega$ elements of type (B2), $\omega-1$ elements of type (A), and $3(\omega-1)$ elements of type (E). The claim follows from the Riemann-Hurwitz formula and Theorem \ref{caratteri}.
\end{proof}

The case $G/G(\Omega)\cong \mathbf{S}_4$ does not occur, since $16\nmid(q^2-1)$.

\begin{proposition}
Let $G\leq M$ be such that $G/G(\Omega)\cong\mathbf{A}_5$, with $n$ even. Then
$$ g(\cH_q/G)= \frac{(q+1)(q-\omega-16)+65\omega-48\epsilon}{120\omega}, $$
where
$$ \epsilon=\begin{cases} \omega & \textrm{if}\quad 5\mid(q-1); \\ 0 & \textrm{if} \quad 5\mid(q+1),\; 5\nmid\omega; \\ q+1 & \textrm{if} \quad 5\mid\omega. \end{cases} $$
\end{proposition}

\begin{proof}
Since $p=2$, the assumption for $n$ to be even is equivalent to require $5\mid(q^2-1)$, so that $\PSL(2,q)$ admits a subgroup isomorphic to $\mathbf{A}_5$ by Theorem \ref{Di}.
We have $G=G(H)\times G(\Omega)$ with $G(H)\cong \mathbf{A}_5$.
By Lemma \ref{classificazione}, $G$ contains: $15$ elements of order $2$, which are of type (C); $20\omega$ elements of order $3$ times a divisor of $q+1$, which are of type (B2); $\omega-1$ nontrivial elements in $G_\Omega$, which are of type (A); $24$ elements of order $5$ in $G(H)$, which are of type (B1) if $5\mid(q+1)$ and of type (B2) if $5\mid(q-1)$; $24(\omega-1)$ elements of order divisible by $5$ in $G\setminus (G(H)\cup G(\Omega))$.
Consider the $24(\omega-1)$ elements $\sigma_i$ of order $5$ in $G\setminus (G(H)\cup G(\Omega))$. If $5\mid(q-1)$, then all $\sigma_i$'s are of type (B2); if $5\mid(q+1)$ and $5\nmid\omega$, then all $\sigma_i$'s are of type (B1).
Suppose that $5\mid\omega$. Let $\sigma_i=\alpha_i\beta_i$ with $\alpha\in G(H)$, $o(\alpha_i)=5$, $\beta_i\in G(\Omega)\setminus\{id\}$.
Using the plane model \eqref{M1} we can assume up to conjugation that $\alpha_i$ and $\beta_i$ stabilize pointwise the reference triangle, $\alpha_i=\diag(\lambda,\lambda^{-1},1)$ where $\lambda^5=1$, and $\beta_i=\diag(\mu,\mu,1)$. The type of $\sigma_i$ is either (A) or (B1), and when $\alpha$ is given there are exactly $2$ choices of $\beta_i$ for which $\sigma_i$ is of type (A), namely $\beta_i=\diag(\lambda,\lambda,1)$ or $\beta_i=\diag(\lambda^{-1},\lambda^{-1},1)$. Then $24\cdot2$ elements are of type (A), and $24(\omega-3)$ are of type (B1).
Now the claim follows from the Riemann-Hurwitz formula and Theorem \ref{caratteri}.
\end{proof}

\begin{proposition}
Let $G\leq M$ be such that $G/G(\Omega)$ is the semidirect product of an elementary abelian $2$-group of order $2^f$ by a cyclic group of order $d$, with $f\leq n$ and $d\mid \gcd(2^f-1,q-1)$. Then
$$ g(\cH_q/G)=\frac{(q+1)(q-\omega-2^f)+\omega (2^f+1)}{2^{f+1}d\omega}. $$
\end{proposition}

\begin{proof}
We have $G=G(H)\times G(\Omega)$ with $G(H)\leq H$ and $G(H)=E_{2^f}\rtimes C_d$, where $E_{2^f}$ is elementary abelian of order $2^f$ and $C_d$ is cyclic of order $d$.
By Lemma \ref{classificazione}, $G$ contains: $2^f-1$ involutions, which are of type (C); $\omega-1$ nontrivial elements of $G_\Omega$, which are of type (A); $(2^f-1)(\omega-1)$ elements of order $2$ times a nontrivial divisor of $q+1$, which are of type (E); $2^f (d-1) \omega$ elements whose order divides $q^2-1$ but does not divide $q+1$, which are of type (B2).
Then the claim follows from the Riemann-Hurwitz formula and Theorem \ref{caratteri}.
\end{proof}

\begin{proposition}\label{conPSL}
Let $G\leq M$ be such that $G/G(\Omega)\cong\PSL(2,2^f)$ with $f\mid n$ and $f>1$. Then
$$ g(\cH_q/G)=\frac{(q+1)\left[q-\omega-2^f(2^f-1)\gcd(2^f+1,\omega)-2^f\right]+(2^f+1)\omega(2^{2f}-2^f+1)}{2^{f+1}(2^f+1)(2^f-1)\omega}$$
if $n/f$ is odd, while
$$g(\cH_q/G)=\frac{(q+1)(q-2^{2f}-\omega)-\omega(2 \cdot 2^{3f}-2^{2f}-2 \cdot 2^{f}-1)}{2^{f+1}(2^f+1)(2^f-1)\omega}+1$$
if $n/f$ is even.
\end{proposition}

\begin{proof}
From Theorem \ref{Di} follows that the elements in $\PSL(2,2^f)$ of order coprime to $|G(\Omega)|$ generate $\PSL(2,2^f)$; hence $G=G(H)\times G(\Omega)$, where $G(H)\cong\PSL(2,2^f)$.
Now we use the order statistics and the subgroup lattice of $\PSL(2,q)$, as decribed in \cite[Chapter II.8]{Hup}.
\begin{itemize}
\item $G(\Omega)$ contains $\omega-1$ elements of type (A).
\item $\PSL(2,2^f)$ contains exactly $(2^f-1)(2^f+1)$ elements of order $2$, which are of type (C). The product of one of them with a nontrivial element of $G(\Omega)$ has order $2$ times a nontrivial divisor of $q+1$; thus, we have $(2^f-1)(2^f+1)(\omega-1)$ elements of type (E).

\item $\PSL(2,2^f)$ contains exactly $\binom{2^f+1}{2}(2^f-2)=\frac{2^f(2^f+1)(2^f-2)}{2}$ nontrivial elements whose order divides $2^f-1$ and hence also $q-1$. The product of one of them with a nontrivial element of $G(\Omega)$ has order a divisor of $q^2-1$ not dividing $q+1$. Thus, we have $\frac{2^f(2^f+1)(2^f-2)}{2}\cdot\omega$ elements of type (B2).

\item $\PSL(2,2^f)$ contains exactly $\frac{2^f(2^{2f}-2^f)}{2}$ nontrivial elements whose order divides $2^f+1$.

Suppose that $n/f$ is odd. Then $2^f+1$ divides $q+1$. Since $H$ contains no elements of type (A), any such element is of type (B1). Together with the identity, they form $\frac{2^{2f}-2^f}{2}$ distinct cyclic groups of order $2^f+1$ which intersect pairwise trivially. Consider a cyclic subgroup $C\leq G(H)$ of order $2^f+1$.
We use the plane model \eqref{M1} of $\cH_q$, and assume up to conjugacy that the self-polar triangle fixed pointwise by $C$ is the reference triangle and the center of the homologies in $G(\Omega)$ is $(0,0,1)$. In this way, $C=\langle\alpha=\diag(\lambda,\lambda^{-1},1)\rangle$ with $o(\lambda)=2^f+1$, while $G(\Omega)=\langle\beta=\diag(\mu,\mu,1)\rangle$ with $o(\mu)=\omega$.
The element $\alpha^i\beta^j$ is either of type (A) or of type (B1); for given $\alpha^i$, $\alpha^i\beta^j$ is of type (A) if and only if $\mu^j=\lambda$ or $\mu^j=\lambda^{-1}$. Hence, there are exactly $\left(\gcd(2^f+1,\omega)-1\right)\cdot2$ elements of type (A) in $C\times G(\Omega)\setminus G(\Omega)$, and exactly $\frac{2^{2f}-2^f}{2}\cdot\left(\gcd(2^f+1,\omega)-1\right)\cdot2$ elements of type (A) in $G\setminus G(\Omega)$.
If $n/f$ is even, then $2^f+1$ divides $q-1$; hence, all the $\frac{2^f(2^{2f}-2^f)\omega}{2}$ elements in this class are of type (B2).
\end{itemize}
In the Riemann-Hurwitz formula $(q+1)(q-2)=|G|(2g(\cH_q/G)-2)+\Delta$, we have by Theorem \ref{caratteri} that
$$\Delta=  (\omega-1)(q+1)+(2^f-1)(2^f+1)(q+2)+(2^f-1)(2^f+1)(\omega-1)\cdot1$$
$$ +\frac{2^f(2^f+1)(2^f-2)\omega}{2}\cdot2+\delta,$$
where
$$\delta=(2^{2f}-2^f)\left(\gcd(2^f+1,\omega)-1\right)\cdot(q+1), $$
if $n/f$ is odd, while
$$\delta=\frac{2^f(2^{2f}-2^f)\omega}{2} \cdot 2,$$
otherwise.
The claim follows by direct computation.
\end{proof}

The case $G/G(\Omega)\cong \PGL(2,2^f)$ is given by Proposition \ref{conPSL}, since $\PGL(2,q)\cong\PSL(2,q)$ when $q$ is even.

\section{New genera for maximal curves}\label{sec:nuovigeneri}

The results of Sections \ref{sec:triangolo} and \ref{sec:polopolare} provide many genera for maximal curves over finite fields. It is possible to compare these values with the ones previously given in \cite{ABB,ATT,BMXY,CO,CO2,CKT1,CKT2,DO,DO2,FG,GGS,GSX}. We list here some new genera,  for some values of $q$.

\begin{center}
\begin{table}[H]
\begin{small}
\caption{New genera $g$ for $\mathbb{F}_{q^2}$-maximal curves} \label{table1}
\begin{center}
\begin{tabular}{|c|c|}
\hline $q$ & $g$ \\
\hline $13$  & $1$\\
\hline $2^5$ & $20$, $55$ \\
\hline $2^7$ & $22$, $133$, $287$, $420$, $903$, $904$ \\
\hline $3^5$ & $10$, $161$, $280$, $590$, $1180$, $2420$ \\
\hline $3^7$ & $91$, $1457$, $24661$, $49595$, $99190$, $198926$ \\
\hline $5^3$ & $17$, $39$, $46$, $63$, $91$, $134$, $210$, $211$, $273$, $274$, $369$, $630$, $631$, $861$ \\
\hline
\end{tabular}
\end{center}
\end{small}
\end{table}
\end{center}

\begin{remark}
Table {\rm \ref{table1}} gives a partial answer to {\rm \cite[Remark 4.4]{ATT}}; namely, Propositions {\rm \ref{indice3caso1}} and {\rm \ref{indice6}} yield examples of $\mathbb{F}_{13^2}$-maximal curves of genus $1$.
\end{remark}

\begin{flushleft}
\end{flushleft}


\begin{thebibliography}{99}

\bibitem{ABB} Anbar, N., Bassa, A., Beelen, P.: {\it A complete characterization of Galois subfields of the generalized Giulietti-Korchm\'aros function field}, Finite Fields Appl. {\bf 49}, 132--142 (2018).

\bibitem{AQ} Abd\'on, M., Quoos, L.: {\it On the genera of subfields of the Hermitian function field}, Finite Fields Appl. {\bf 10}, 271--284 (2004).

\bibitem{ATT} Arakelian, N., Tafazolian, S., Torres, F.: {\it On the spectrum for the genera of maximal curves over small fields}, Adv. Math. Commun. {\bf 12} (1), 143--149 (2018).

\bibitem{BMXY} Bassa, A., Ma, L., Xing, C., Yeo, S.L.: {\it Toward a characterization of subfields of the Deligne-Lusztig function fields}, J. Combin. Theory Series A {\bf 120}, 1351--1371 (2013).

\bibitem{BM} Beelen, P., Montanucci, M.: {\it A new family of maximal curves}, preprint, arXiv:1711.02894.

\bibitem{CO} \c{C}ak\c{c}ak, E., \"Ozbudak, F.: {\it Subfields of the function field of the Deligne-Lusztig curve of Ree type}, Acta Arith. {\bf 115} (2), 133--180 (2004).

\bibitem{CO2} \c{C}ak\c{c}ak, E., \"Ozbudak, F.: {\it Number of rational places of subfields of the function field of the Deligne-Lusztig curve of Ree type}, Acta Arith. {\bf 120} (1), 79--106 (2005).

\bibitem{CKT1} Cossidente, A., Korchm\'aros, G., Torres, F.: {\it On curves covered by the Hermitian curve}, J. Algebra {\bf 216} (1), 56--76 (1999).

\bibitem{CKT2} Cossidente, A., Korchm\'aros, G., Torres, F.: {\it Curves of large genus covered by the Hermitian curve}, Comm. Algebra {\bf 28}, 4707--4728 (2000).

\bibitem{DO} Danisman, Y., \"Ozdemir, M.: {\it On subfields of GK and generalized GK function fields}, J. Korean Math. Soc. {\bf 52} (2), 225--237 (2015).

\bibitem{DO2} Danisman, Y., \"Ozdemir, M.: {\it On the genus spectrum of maximal curves over finite fields}, J. Discr. Math. Sc. and Crypt. {\bf 18} (5), 513--529 (2015).

\bibitem{DL} Deligne, P., Lusztig, G.: {\it Representations of reductive groups over finite fields}, Ann. of Math. {\bf 103} (1), 103--161 (1976).

\bibitem{D} Dickson, L.E.: {\it Linear Groups with an Exposition of the Galois Field Theory}, Teubner, Leipzig (1902).

\bibitem{DM} Duursma, I., Mak, K.H.: {\it On maximal curves which are not Galois subcovers of the Hermitian curve}, Bull. Braz. Math. Soc. (N.S.) {\bf 43} (3), 453--465 (2012).

\bibitem{FG} Fanali, S., Giulietti, M.: {\it Quotient curves of the GK curve}, Adv. Geom. {\bf 12}, 239--268 (2012).

\bibitem{FT} Fuhrmann, R., Torres, F.: {\it On Weierstrass points and optimal curves}, Rend. Circ. Mat. Palermo Suppl. \textbf{51} (Recent Progress in Geometry, Ballico E, Korchm\'aros G, (Eds.)), 25--46 (1998).

\bibitem{G} Garcia, A.: {\it Curves over finite fields attaining the Hasse-Weil upper bound}, in: European Congress of Mathematics, vol. II (Barcellona 2000), Progr. Math. \textbf{202}, Birkh\"auser, Basel,  199--205 (2001).

\bibitem{G2} Garcia, A.: {\it On curves with many rational points over finite fields}, in: Finite Fields with Applications to Coding Theory, Cryptography and Related Areas, Springer, Berlin, 152--163 (2002).

\bibitem{GGS} Garcia, A., G\"uneri, C., Stichtenoth, H.: {\it A generalization of the Giulietti-Korchm\'aros maximal curve}, Adv. Geom. {\bf 10} (3), 427--434 (2010).

\bibitem{GS} Garcia, A., Stichtenoth, H.: {\it Algebraic function fields over finite fields with many rational places}, IEEE Trans. Inform. Theory {\bf 41}, 1548--1563 (1995).

\bibitem{GS3} Garcia, A., Stichtenoth, H.: {\it A maximal curve which is not a Galois subcover of the Hermitian curve}, {Bull. Braz. Math. Soc. (n.S.)} {\bf 37}, 139--152 (2006).

\bibitem{GSX} Garcia, A., Stichtenoth, H., Xing, C.P.: {\it On subfields of the Hermitian function field}, Compositio Math. {\bf 120}, 137--170 (2000).

\bibitem{GHKT2} Giulietti, M., Hirschfeld, J.W.P., Korchm\'aros, G., Torres, F.: {\it A family of curves covered by the Hermitian curve}, S\'emin. Congr. {\bf 21}, 63--78 (2010).

\bibitem{GK} Giulietti, M., Korchm\'aros, G.: {\it A new family of maximal curves over a finite field}, Math. Ann. {\bf 343}, 229--245 (2009).

\bibitem{GMZ}  Giulietti, M., Montanucci, M., Zini, G.: {\it On maximal curves that are not quotients of the Hermitian curve}, Finite Fields Appl. \textbf{41}, 72--88 (2016).

\bibitem{GQZ} Giulietti, M., Quoos, L., Zini, G.: {\it Maximal curves from subcovers of the GK-curve}, J. Pure Appl. Algebra {\bf 220}, 3372--3383 (2016).

\bibitem{GOS} G\"uneri, C., \"Ozdemir, M., Stichtenoth, H.: {\it The automorphism group of the generalized Giulietti-Korchm\'aros function field}, Adv. Geom. {\bf 13}, 369--380 (2013).

\bibitem{H} Hartley, R.W.: {\it Determination of the ternary collineation groups whose coefficients lie in the $GF(2^n)$}, Ann. of Math. Second Series {\bf 27} (2), 140--158 (1925).

\bibitem{HKT} Hirschfeld, J.W.P., Korchm\'aros, G., Torres, F.: {\it Algebraic Curves over a Finite Field}, Princeton Series in Applied Mathematics, Princeton (2008).

\bibitem{HO} Hoffer, A.R.: {\it On unitary collineation groups}, J. Algebra {\bf 22}, 211--218 (1972).

\bibitem{HP} Hughes, D.R., Piper, F.C.: {\it Projective Planes}. Graduate Text in Mathematics {\bf 6}, Springer, Berlin, xii+793 pp. (1973).

\bibitem{Hup} Huppert, B.: {\it Endliche Gruppen I}. Grundlehren der Mathematischen Wissenschaften {\bf 134}, Springer, Berlin (1967).

\bibitem{KL} Kleidman, P., Liebeck, M.: {\it The Subgroup Structure of the Finite Classical Groups}, London Mathematical Society, Lecture Note Series {\bf 129}, Cambridge University Press (1990).

\bibitem{KS} Kleiman, S.L. : \textit{Algebraic cycles and the Weil conjectures}, in: Dix expos\'es sur la cohomologie des sch\'emas, in: Adv. Stud. Pure Math. \textbf{3}, 359-386 (1968).

\bibitem{L} Lachaud, G.: {\it Sommes d'Eisenstein et nombre de points de certaines courbes alg\'ebriques sur les corps finis}, C.R. Acad. Sci. Paris \textbf{305}, S\'erie I, 729--732 (1987).

\bibitem{Mak} Mak, K.H.: {\it On Congruence Function Fields with many rational places}, PhD Thesis,
www.ideals.illinois.edu/bitstream/handle/2142/34193/Mak\textunderscore KitHo.pdf?sequence=1.

\bibitem{M} Mitchell, H.H.: {\it Determination of the ordinary and modular ternary linear groups}, Trans. Amer. Math. Soc. {\bf 12} (2), 207--242 (1911).

\bibitem{MZultimo} Montanucci, M., Zini, G.: {\it Quotients of the Hermitian curve from subgroups of ${\rm PGU}(3,q)$ without fixed points or triangles}, preprint, arXiv: 1804.03398.

\bibitem{MZ} Montanucci, M., Zini, G.: {\it On the spectrum of genera of quotients of the Hermitian curve}, Comm. Algebra (2018), DOI 10.1080/00927872.2018.1455100.

\bibitem{MZRS} Montanucci, M., Zini, G.: {\it Some Ree and Suzuki curves are not Galois covered by the Hermitian curve}, Finite Fields Appl. {\bf 48}, 175--195 (2017).

\bibitem{RobinHood} Re, R.: {\it Supersingular quotients of Fermat curves}, Finite Fields Appl. {\bf 15} (4), 450-467 (2009).

\bibitem{Sti} Stichtenoth, H.: {\it Algebraic function fields and codes}, 2nd edn. Graduate Texts in Mathematics {\bf 254}, Springer, Berlin (2009).

\bibitem{Suzuki} Suzuki, M.: {\it Group Theory I, II}, Springer-Verlag, Berlin (1986).

\bibitem{TTT} Tafazolian, S., Teher\'an-Herrera A., Torres, F.: {\it Further examples of maximal curves which cannot be covered by the Hermitian curve}, J. Pure Appl. Algebra {\bf 220}, 1122--1132 (2016).

\bibitem{vdG} van der Geer, G.: {\it Curves over finite fields and codes}, In: {\it European Congress of Mathematics}, vol. II (Barcellona 2000), Progr. Math. \textbf{202}, Birkh\"auser, Basel, 225--238 (2001).

\bibitem{vdG2} van der Geer, G.: {\it Coding theory and algebraic curves over finite fields: a survey and questions}, In: { Applications of Algebraic Geometry to Coding Theory, Physics and Computation, NATO Sci. Ser. II Math. Phys. Chem.} \textbf{36}, Kluwer, Dordrecht, 139--159 (2001).

\bibitem{WALL} Wall, C.T.C.: {\it On the structure of finite groups with periodic cohomology. (English summary) Lie groups: structure, actions, and representations}, Progr. Math. {\bf 306}, Birkh\"auser/Springer, New York, 381--413 (2013).

\end{thebibliography}
\end{document}